\documentclass[10pt]{article}
\usepackage[normalem]{ulem}
\usepackage{hyperref}
\usepackage{amsmath,dsfont}
\usepackage{amsthm}
\usepackage{enumerate}
\usepackage{amssymb}
\usepackage[latin1]{inputenc}
\usepackage{color}
\usepackage{graphicx}
\usepackage{subcaption}
\usepackage{caption}
\begin{document}
\newcommand{\R}{{\bf R}}
\newcommand{\K}{{\bf K}}
\newcommand{\C}{{\bf C}}
\newcommand{\sequ}[2]{(#1_{#2})_{#2\geq 1}}
\newcommand{\rsequ}[1]{(#1_n)_{n=1,2,\dots}}
\newcommand{\seq}[1]{(#1(n))_{n=0,1,\dots}}
\newcommand{\rseq}[1]{(#1_n)_{n=0,1,\dots}}
\newcommand{\css}[1]{{\color{red}#1}}
\newcommand{\cs}[1]{{\color{black}#1}}
\newcommand{\csl}[1]{{\color{black}#1}}
\newtheorem{thm}{Theorem}[section]
\newtheorem{example}{Example}
\newtheorem{defn}[thm]{Definition}
\newtheorem{cor}[thm]{Corollary}
\newcommand{\rtref}[1]{{\rm \ref{#1}}}
\newcommand{\rmref}[1]{{\rm (\ref{#1})}}
\newcommand{\rmcite}[1]{{\rm \cite{#1}}}
\newcommand{\rmciter}[2]{{\rm \cite[#1]{#2}}}
\newtheorem{lem}[thm]{Lemma}
\newtheorem{prop}[thm]{Proposition}
\newtheorem{rem}[thm]{Remark}

	\title {Random positive linear operators and their applications to nonparametric statistics\footnote{Accepted for publication in Test}}
	\author{José A. Adell,\footnote{e-mail: adell@unizar.es, Orcid ID: 0000-0001-8331-5160 } \ \ J. T. Alcalá\footnote{e-mail: jtalcala@unizar.es, Orcid ID: 0000-0001-7549-8825} \\ and  \ \  C. Sangüesa\footnote{Corresponding author. e-mail: csangues@unizar.es, Orcid ID: 0000-0002-7099-7665}
		\\{$^{\dag} $ \small Department of Statistical Methods and IUMA, University of Zaragoza, Zaragoza, 50009,  SPAIN}
	}
	
	\date{}

	\begin{titlepage}
		\setcounter{page}{1} \maketitle
		
		\bigskip \bigskip
		\begin{abstract}
			We outline a general procedure on how to apply random positive linear operators in nonparametric estimation.  As a consequence, we give explicit confidence bands and intervals for a distribution function $F$ concentrated on $[0,1]$ by means of random Bernstein polynomials, and for the derivatives of $F$ by using random Bernstein-Kantorovich type operators.  In each case, the lengths of such bands and intervals depend upon the degree of smoothness of $F$ or its corresponding derivatives, measured in terms of appropriate moduli of smoothness.  In particular, we estimate the uniform distribution function by means of a random polynomial of second order. This estimator is much simpler and performs better than the classical uniform empirical process used in the celebrated Dvoretzky-Kiefer-Wolfowitz inequality.
			
		\end{abstract}

		
		\bigskip \bigskip

		\bigskip \noindent

	\end{titlepage}

	\maketitle
	
	
	\section{Introduction}
	The Bernstein polynomials are the most paradigmatic example of positive linear operators in approximation theory.  Their approximation and preservation properties are well known (see, for instance, the monograph by Bustamante \cite{bubern}).  The approximation properties of Bernstein-Kantorovich type operators needed in this paper are collected in Section \ref{seappr}. 
	
	Since the pioneering work by Vitale \cite{viaber}, a great deal of attention has been devoted to obtain applications of Bernstein polynomials in nonparametric statistics.  Indeed, Bernstein estimators have been used to estimate distribution functions and densities (cf. Vitale \cite{viaber}, Babu et al. \cite{baappl},  Bouezmarni and Rolin \cite{borobe}, Leblanc \cite{lechun,leones}, Dutta \cite{dudist}, Ouimet \cite{ouassy}, and Wang and Lu \cite{waluap}), conditional distribution functions and densities (cf. Janssen et al. \cite{jaswbe} and Belalia et al. \cite{bebosm}), copula functions and densities (cf. Janssen et al. \cite{jaswbe,jaswan}), regression functions (cf. Tenbusch \cite{tenonp} and Slaoui \cite{slbern}), regression quantile functions (Janssen et al. \cite{jaswbe} and Khardani \cite{khbern}), and cross ratio functions (Abrams et al. \cite{abjano} and Abrams et al. \cite{abbern}), among many others. An excelent review of the literature on the subject can be found in Ouimet \cite{ouonth}.
	
	The usual approach in nonparametric estimation by means of random Bernstein polynomials can be illustrated by considering the estimation of distribution functions (see, for instance, Leblanc \cite{leones}, Ouimet \cite{ouassy}, and Wang and Lu \cite{waluap}).  Suppose that $X$ is a $[0,1]$-valued random variable having an unknown continuous distribution function $F$, and that $\mathbb{Y}_n=(Y_n(x),\ x\in [0,1] )$ is its corresponding empirical distribution function of size $n$.  The stochastic process $\mathbb{Y}_n$, the natural estimator of $F$, has discontinuous paths.  For this reason, we construct the random Bernstein polynomial of degree $m$, namely, $B_m(Y_n;x)$ acting on $Y_n$, having infinitely differentiable paths (see Section \ref{sedete} for precise definitions).  Under the assumption that $F$ is twice continuously differentiable and that $0<F(x)<1,\ x\in (0,1)$, we have
	\[n^{1/2}\left(B_m(Y_n;x)-F(x)\right) \overset{({\cal L})}{\longrightarrow} N(0,F(x)(1-F(x))), \quad x\in (0,1),\]
	whenever $m/\sqrt{n}\rightarrow \infty$, as $m,n \rightarrow \infty$, where $\overset{({\cal L})}{\longrightarrow}$ denotes convergence in law.  This allows us to provide asymptotic confidence intervals for $F(x),\ x\in (0,1)$.  The optimal degree $m=m(n)$ of the Bernstein polynomial is that minimizing the mean square error (MSE) or the mean integrated square error (MISE), respectively defined as
	\[MSE=\mathbb{E}\left(B_m(Y_n;x)-F(x)\right)^2, \quad \hbox{and} \quad MISE=\mathbb{E}\int_0^1\left(B_m(Y_n;x)-F(x)\right)^2dx.\]
	It turns out that, in both cases, the order of magnitude of $m$ is $m\sim n^{2/3}$ (a random degree is chosen in Dutta \cite{dudist}).
	
\cs{In other words, for a given confidence level $1-\alpha$, with $\alpha\in(0,1)$ and under the preceding assumptions on $F$, we asymptotically have for any $x\in(0,1)$
	\begin{equation}
		P\left(\left|B_m(Y_n;x)-F(x)\right|\geq z_{\alpha/2}\frac{ \sqrt{F(x)(1-F(x))}}{\sqrt{n}}\right)\sim \alpha, \quad  m\sim  n^{2/3},\quad n\rightarrow \infty. \label{assint}
	\end{equation}
	In this paper, we assume that $F$ is simply continuous and that $0<F(x)<1$, $x\in(0,1)$.  For any $x\in(0,1)$, we construct explicit confidence intervals of the form
	\begin{equation}
	P\left(\left|B_m(Y_n;x)-F(x)\right|\geq r_{\alpha}\frac{ \sqrt{F(x)(1-F(x))}}{\sqrt{n}}\right)\leq \alpha, \quad  n\geq n_{0}(\alpha,x), \label{conint}
\end{equation}
as well as confidence bands 
	\begin{equation}
	P\left(\left\|B_m(Y_n;x)-F(x)\right\|>\frac{ \widetilde{r}_{\alpha}}{\sqrt{n}}\right)\leq \alpha,  \label{conban}
\end{equation}
where $ r_{\alpha}$ and $\widetilde{r}_{\alpha}$ are positive constants only depending upon $\alpha$ (see Corollary \ref{coink0} and Corollary \ref{coubof}, respectively).  Then, we determine the minimum degree $m$ of the Bernstein polynomial, i.e., the simplest Bernstein estimator $B_m(Y_n;x)$ satisfying (\ref{conint}) and (\ref{conban}).  In each case, this degree $m$ essentially depends upon the rate of convergence of $\left|B_m(F;x)-F(x)\right|$ or  $\left\|B_m(F;x)-F(x)\right\|$ towards $0$.  In turn, such rates depend on the degree of smoothness of $F$, measured in terms of appropriate moduli of smoothness (see Section \ref{seappr}).  }
	
	In the particular case of the uniform distribution, the minimum degree $m$ of the Bernstein polynomial is $m=2$.  This leads us to the following striking results (see Corollary \ref{counup} and Corollary \ref{counba}, respectively)
	\[P\left(\left|B_2(Y_n;x)-x\right|\geq 4 x(1-x)\delta\right)\leq 2  exp\left(-8n\delta^2\left(1-\frac{8}{3}\delta\right)\right),\ x\in(0,1),\ 0<\delta\leq 1/8,\]
	as well as
	\[P\left(\left\|B_2(Y_n;x)-x\right\|>\delta\right)\leq 2  e^{-8n\delta^2},\quad \delta>0.\]
	In other words, $B_2(Y_n;x)$ is a random polynomial of degree $2$ only involving a random variable having the binomial distribution with parameters $n$ and $1/2$.  This estimator is much simpler and performs in a more efficient way than the classical uniform empirical process used in the Dvoretzky-Kiefer-Wolfowitz (D-K-W) inequality (cf. Dvoretzky et al. \cite{dvkias} and Massart \cite{mathet}).
	
	Formulas (\ref{conint}) and (\ref{conban}) are also extended to the derivatives $F^{(k)}$ of the distribution function $F$ (see Sections \ref{seconf}-\ref{secoid}).  In fact, we use random Bernstein polynomials to estimate $F$, and random Bernstein-Kantorovich type operators to estimate $F^{(k)},\ k\geq 1$.  This kind of results has several applications in statistical data analysis, such as local minima/maxima, optimal bandwidth in kernel density estimation and so on (see, for instance, Sasaki et al. \cite{sadire} or, in a multivariate setting, Chac\'{o}n and Duong \cite[Chap. 5 and 6]{chdumu} and the references therein).
	
	To keep the paper in a moderate size, we only consider Bernstein estimators of $F$ and its derivatives.  However, in Section \ref{sedete}, we outline a unified procedure on how to apply general random positive linear operators in nonparametric estimation.  In this regard, classical kernel estimators can be viewed as estimators built up from random Weierstrass type operators.
	
	\cs{In Section \ref{sesimu}, we illustrate some of the results in Sections \ref{seconf}-\ref{secoid} by considering simulations referring to a family of specific distribution functions with different degrees of smoothness.  The proofs of the results in Sections \ref{seconf}-\ref{sesimu} have been moved to Appendix \ref{seappe}.  In this regard, the proofs concerning confidence bands are mainly based on the D-K-W inequality, whereas those concerning conficence intervals rely mainly on the concentration inequality given in Lemma \ref{lecowt} in the Appendix.  Finally, a summary of the contributions of ths paper is provided in Section \ref{seconcl}.}
	\section{Deterministic and random positive linear operators} \label{sedete}
	Throughout this paper, the following notations and assumptions will be used. Let $\mathbb{N}$ be the set of positive integers and $\mathbb{N}_0=\mathbb{N}\cup \{0\}$. We assume that $n,m\in \mathbb{N}$, with $n,m\geq 2$.  The usual supremum norm is denoted by $\|\cdotp\|$.  Given a closed real interval $I$, we denote by $C^{k}(I)$ the set of all $k$ times continously differentiable functions defined on $I$, $k\in \mathbb{N}_0$, and set $C(I)=C^{0}(I)$.  \cs{Any stochastic process considered has right-continuous paths.  Finally, for the sake of clarity, we consider two probability spaces $(\Omega,{\cal F}, P)$ and $(\widetilde{\Omega},\widetilde{\cal F}, \widetilde{P})$, together with their product  $(\Omega\times \widetilde{\Omega},{\cal F}\otimes \widetilde{\cal F}, P\otimes \widetilde{P})$. Expectations in these three spaces are denoted by $\mathbb{E}$, $\widetilde{\mathbb{E}}$ and $\mathbb{E}\times\widetilde{\mathbb{E}}$, respectively.  Suppose that $Y$ (resp. $\widetilde{Z}$) is a real-valued random variable defined on $\Omega$ (resp. $\widetilde{\Omega}$ ). We can always assume that these two random variables are defined on  $\Omega\times \widetilde{\Omega}$ via the formulas
	\begin{equation}
		Y_1(\omega,\widetilde{\omega})=Y(\omega),\quad \widetilde{Z}_1(\omega,\widetilde{\omega})=\widetilde{Z}( \widetilde{\omega}),\quad (\omega,\widetilde{\omega})\in\Omega\times\widetilde{\Omega}.\label{varpro}
		\end{equation}
	It turns out that $Y_1$ and $\widetilde{Z}_1$ are independent.  Actually, for any Borel sets $A$ and $B$, we have $Y_1^{-1}(A)=Y^{-1}(A)\times \widetilde{\Omega}$ and $\widetilde{Z}_1^{-1}(B)=\Omega\times \widetilde{Z}^{-1}(B)$, thus implying that
	\begin{align}&P\times \widetilde{P}\left(Y_1^{-1}(A)\cap \widetilde{Z}_1^{-1}(B)\right)=P\times \widetilde{P}\left(Y^{-1}(A)\times \widetilde{Z}^{-1}(B)\right) \nonumber\\&= P\left(Y^{-1}(A)\right) \widetilde{P} \left(\widetilde{Z}^{-1}(B)\right) = P\times \widetilde{P}\left(Y_1^{-1}(A)\right) P\times \widetilde{P} \left(\widetilde{Z}_1^{-1}(B)\right).\label{indpro} \end{align} }
	
	\cs{Let $\widetilde{\mathbb{Z}}_m=(\widetilde{Z}_m(x),\ x\in I)$ be a stochastic process such that $\widetilde{Z}_m(x)$ takes values in $I$, for any $x\in I$, and satisfies
	\begin{equation}
		\widetilde{Z}_m(x)\xrightarrow{({\cal L})}x, \quad x\in I, \quad m\rightarrow \infty.\label{convez}
	\end{equation}}
	Associated to  $\widetilde{\mathbb{Z}}_m$, we define the positive linear operator
\cs{	\begin{equation}
		L_m(f;x)=\widetilde{\mathbb{E}}f(\widetilde{Z}_m(x))=\int_{\mathbb{R}}f(\theta)d\widetilde{F}_{m,x}(\theta),\quad x\in I,\quad f\in {\cal M}(I),\label{dflino}
	\end{equation}
	where $\widetilde{F}_{m,x}$ is the distribution function of $\widetilde{Z}_m(x)$ and ${\cal M}(I)$} is the set of all real measurable functions defined on $I$ for which the preceding expectations exist. We assume that
	\begin{equation}
		L_m( {\cal M}(I))\subseteq C^k(I),\quad k\in \mathbb{N}_0.\label{operco}
	\end{equation}
\cs{The Continuous Mapping Theorem and assumption (\ref{convez}) guarantees} that $L_m(f;x)\rightarrow f(x)$, as $m\rightarrow \infty$, for any bounded function $f\in C(I)$. By (\ref{operco}), $L_m(f;x)$ is a smooth approximant of $f(x)$.
	
	Before going further, we consider two well known examples of operators $L_m$ satisfying (\ref{convez})-(\ref{operco}) (see, for instance \cite{adcabe}, for further examples).
	
	\begin{example}[{{\bf Bernstein polynomials.} $I=[0,1]$}] \cs{Let $x\in [0,1]$.} Let $(\widetilde{U}_j)_{j\geq 1}$ be a sequence of independent copies of a random variable $\widetilde{U}$ uniformly distributed on $[0,1]$ and consider the uniform empirical process of size $m$
		\begin{equation}
			\widetilde{Z}_m(x) =\frac{	\widetilde{S}_m(x)}{m}, \qquad 		\widetilde{S}_m(x)=\sum_{j=1}^{m}1_{[0,x]}(	\widetilde{U}_j), \label{empin}
		\end{equation}
		where $1_A$ stands for the indicator function of the set $A$.	The $m$-th Bernstein polynomial of a function $f\in {\cal M}([0,1])$ is defined by
		\begin{equation}B_m(f;x)=\widetilde{\mathbb{E}}f\left(\frac{	\widetilde{S}_m(x)}{m}\right)=\sum_{k=0}^m f\left(\frac{k}{m}\right)p_{m,k}(x),\quad p_{m,k}(x)={m \choose k}x^k(1-x)^{m-k}. \label{berpol}
		\end{equation} \label{exampb}
	\end{example}

	\begin{example}[{\bf Weierstrass-type operators}. $I=\mathbb{R}$]\label{exampl2} \cs{Let $x\in \mathbb{R}$.}  Let $\widetilde{W}$ be \cs{an arbitrary real} valued random variable \cs{with distribution function $\widetilde{F}$} and let $\left(h_m\right)_{m\geq 1}$ be a sequence of positive real numbers converging to $0$. The Weierstrass-type operators are defined by
	\cs{	\[W_m(f;x)=\mathbb{\widetilde{E}}f(x+h_m\widetilde{W})=\int_{\mathbb{R}}f(x+h_m\theta)d\widetilde{F}(\theta),\quad f\in {\cal M}(\mathbb{R}).  \]}If $\widetilde{W}$ has the standard normal distribution and $h_m=1/\sqrt{m}$,  $W_m$ is called the $m$-th Weierstrass operator.  \cs{In such a case, denote by $\eta(\theta)$ the standard normal density.  Recall that the Hermite polynomials $(H_k(\theta))_{k\geq 0}$ are defined via the Rodrigues formula 
		\[H_k(\theta)=(-1)^k \frac{\eta^{(k)}(\theta)}{\eta(\theta)},\quad \theta\in \mathbb{R},\quad k\in \mathbb{N}_0.\]
		Let $k\in \mathbb{N}$ and $x\in \mathbb{R}$.  The derivatives of the Weierstrass operator have the form
		\begin{equation}
			W_m^{(k)}(f;x)=\frac{1}{h_m^k}\widetilde{E}f\left(x+h_m\widetilde{W}\right)H_k(\widetilde{W}),\quad  f\in {\cal M}(\mathbb{R}) ,\label{dewedi}
			\end{equation} 
		or
			\[
			W_m^{(k)}(f;x)=\widetilde{E}f^{(k)}\left(x+h_m\widetilde{W}\right),\quad f\in C^k(\mathbb{R}).
			\]  
			Such formulas are analogous to the expressions given in (\ref{debedi}) and (\ref{debeka}) below for the Bernstein polynomials.  As a consequence of (\ref{dewedi}), note that we have the Wiener expansion
			\[W_m(f;x)=\sum_{k=0}^\infty \frac{1}{k!}\mathbb{\widetilde{E}}f(h_m\widetilde{W})H_k(\widetilde{W})\left(\frac{x}{h_m}\right)^k,\quad |x|<h_m,\] 
			whenever $\mathbb{\widetilde{E}}f^2(h_m\widetilde{W})<\infty.$  }\label{exampw}
	\end{example}
	
	Roughly speaking, if in formula (\ref{dflino}), we choose a random function $f$ independent of $\widetilde{Z}_m$, we obtain a random positive linear operator.  More precisely, let $\mathbb{Y}_n=\left(Y_n(x),\ x\in I\right)$ be a stochastic process.  \cs{As follows from (\ref{varpro}) and (\ref{indpro}), no generality is lost if we assume that the stochastic processes $\mathbb{Y}_n$ and  $\mathbb{\widetilde{Z}}_m$, as defined on $\Omega\times \widetilde{\Omega}$, are independent. We then consider the subordinated stochastic process $\left(Y_n(\mathbb{\widetilde{Z}}_m(x)),\ x \in I\right)$ and define the random positive linear operator 
	\begin{equation}
		L_m( Y_n;x)=\widetilde{\mathbb{E}}Y_n(\widetilde{Z}_m(x))=\int_{\mathbb{R}}Y_n(\theta)d\widetilde{F}_{m,x}(\theta),\quad x\in I,\label{dfranl}
		\end{equation}
		where $\widetilde{F}_{m,x}(\theta)$ is the distribution function of $\widetilde{Z}_m(x)$. In other words, we have replaced in (\ref{dflino}) the deterministic function $f(\theta)$ by the random function $Y_n(\theta)$.  Note that the paths of $	L_m( Y_n;x)$ are in $C^k(I)$,\  $k\in \mathbb{N}$, as follows from (\ref{operco}).}
	
	One of the main motivations of definition (\ref{dfranl}) comes from the following estimation problem.  Assume that the process $\mathbb{Y}_n$ is integrable. and that
	\begin{equation} \mathbb{E}Y_n(x)=\mu(x)\quad x\in I,\label{expecc}\end{equation}
	as well as 
	\begin{equation}
		Y_n(x)\xrightarrow{({\cal L})}\mu(x), \quad x\in I, \quad n\rightarrow \infty, \label{convey}
	\end{equation}
	where $\mu\in C(I)$ is unknown. Instead of using the process $\mathbb{Y}_n$  as an estimator of $\mu$, we consider the following approximation
\cs{	\begin{align}
		&L_m(Y_n;x)-\mu(x)=\int_{\mathbb{R}}Y_n(\theta)d\widetilde{F}_{m,x}(\theta)-\mu(x)\nonumber \\
		&=\int_{\mathbb{R}}\left(Y_n(\theta)-\mu(\theta)\right)d\widetilde{F}_{m,x}(\theta)-\int_{\mathbb{R}}\left(\mu(\theta)-\mu(x)\right)d\widetilde{F}_{m,x}(\theta) \nonumber \\&=L_m(Y_n-\mu;x)+L_m(\mu;x)-\mu(x),\quad x\in I.\label{decomp}
	\end{align}
	The term $L_m(\mu;x)-\mu(x)$ is }deterministic and can be estimated using approximation-theoretic techniques based on appropriate moduli of smoothness. By assumption (\ref{operco}), the random term $L_m( Y_n-\mu;x)$ has paths in $C^k(I)$. This implies, under the assumption that $\mu\in C^k(I)$, for some $k\in \mathbb{N}_0$, that we can differentiate in (\ref{decomp}) to obtain the approximation 
	\begin{equation}
		L_m^{(k)}(Y_n;x)-\mu^{(k)}(x)=L_m^{(k)}( Y_n-\mu;x)+L_m^{(k)}(\mu;x)-\mu^{(k)}(x),\quad x\in I,\label{decomd}
	\end{equation}
	thus giving a procedure to estimate $\mu^{(k)}$ by using the smooth estimator $L_m^{(k)}(Y_n;x)$.
	
	In this paper, attention is focused on the case in which $\mu=F\in C[0,1]$ is the distribution function of a certain random variable $X$ taking values in $[0,1]$, and $\mathbb{Y}_n$ is its corresponding empirical process. Specifically, let $\left(X_r\right)_{r\geq 1}$ be a sequence of independent copies of $X$ and define
	\begin{equation}
		Y_n(x)=\frac{1}{n}\sum_{r=1}^{n}1_{(-\infty,x]}(X_r), \quad x\in [0,1]. \label{pampiri}
	\end{equation}
	To rewrite this process,  let $\left(U_r\right)_{r\geq 1}$ be a sequence of independent copies of a random variable $U$ uniformly distributed on $[0,1]$ and denote
	\begin{equation}
		S_n(x)=\sum_{r=1}^{n}1_{[0,x]}(U_r), \quad x\in [0,1]. \label{empiri}
	\end{equation}
	Let 
	\[	
	F^{\star}(u)=\inf\{x\in \mathbb{R},\  F(x)\geq u\},\quad u\in(0,1),\]
	be the quantile function of $X$. Since $F^{\star}(U) \overset{({\cal L})}{=}X$, where $ \overset{({\cal L})}{=}$ stands for equality in law, we can write
	\begin{equation}
		Y_n(x) \overset{({\cal L})}{=}\frac{S_n(F(x))}{n},\quad x\in [0,1]. \label{eqlaym}
	\end{equation}
	By Example \ref{exampb} and (\ref{dfranl}), the random Bernstein polynomials associated to $\mathbb{Y}_n$ are
	\begin{equation}
		B_m(Y_n;x)=\frac{1}{n}\widetilde{\mathbb{E}}S_n\left(F(\widetilde{S}_m(x)/m)\right)=\frac{1}{n}\sum_{k=0}^m S_n\left(F(k/m)\right)p_{m,k}(x), \quad x\in[0,1].\label{rabern} 
	\end{equation}
	If the random variable $X$ takes values in $[0,\infty)$ or $\mathbb{R}$, we should consider the random Sz\`asz-Mirakyan operator (as used in Ouimet \cite{ouonth} and Hanebeck and Klar \cite{haklsm}) or the random Weierstrass-type operator associated to $\mathbb{Y}_n$, respectively. For instance, in this last case, let $\mathbb{Y}_n$ be as in (\ref{pampiri}), with $x\in \mathbb{R}$. \cs{As in Example \ref{exampw}, let  $\widetilde{F}(\theta)$ be the distribution function of $\widetilde{W}$. By (\ref{dfranl}), the associated random Weierstrass-type operator applied to $-\widetilde{W}$ takes on the form 
	\begin{align}
	&W_m(Y_n;x)=\widetilde{E}Y_n(x-h_m\widetilde{W})=\int_{\mathbb{R}}Y_n(x-h_m\theta)d\widetilde{F}(\theta) \nonumber\\
&=	\frac{1}{n}\sum_{r=1}^{n}\int_{\mathbb{R}}1_{(-\infty,x-h_m\theta]}(X_r)d\widetilde{F}(\theta)=\frac{1}{n}\sum_{r=1}^{n}\int_{\mathbb{R}}1_{(-\infty,\left(x-X_r\right)/h_m]}(\theta)d\widetilde{F}(\theta)\nonumber\\
&=\frac{1}{n}\sum_{r=1}^{n}\widetilde{F}\left(\frac{x-X_r}{h_m}\right),  \quad x\in \mathbb{R}. \label{ranwei}\end{align}} This expression gives us the classical kernel estimator of $F$, the distribution function of $X$ (cf. Tsybakov \cite{tsintr}).
	
	Many positive linear operators, such as those considered in the examples above, preserve monotonicity, that is, $L_m(f,x)$ is nondecreasing if $f(x)$ is so. By (\ref{dflino}), this is equivalent to the fact that the stochastic process $(\widetilde{Z}_m(x),\ x\in I)$ has nondecreasing paths (cf. Shaked and Shanthikumar \cite[Chap. 1]{shshst}). If this is the case, and  whenever  $\mathbb{Y}_n$ is an empirical process, and $\mu=F$ is a distribution function, then formulas (\ref{decomp}) and (\ref{decomd}) approximate $F$ and $F^{(1)}$ by smooth random distribution functions and densities, respectively. 
	\section{Bernstein-Kantorovich type operators} \label{seappr}
	In this section, we consider some approximation properties of Bernstein and Bernstein-Kantorovich type operators, which will be used in estimating distribution functions and their derivatives, respectively. From now on, we assume that all of the random variables appearing under the same expectation sign are mutually independent. Also, we denote by $\sigma^2(x)=x(1-x),\ x\in [0,1]$.
	
	Let $k\in \mathbb{N}_0$, $h\geq 0$, and $0\leq x\leq x+k h\leq 1$. Recall that the $k$-th forward differences of a function $f$ defined on $[0,1]$ at step $h$ are recursively defined as
	\begin{equation}
		\Delta_h^0 f(x)=f(x), \ 	\Delta_h^1f(x)=f(x+h)-f(x),\ \Delta_h^kf(x)=\Delta_h^{k-1}\left(\Delta_h^1f\right)(x), \ k\geq 2, \label{eqfwdi}
	\end{equation}
	or, equivalently, by
	\begin{equation}
		\Delta_h^kf(x)=\sum_{j=0}^k {k \choose j}(-1)^{k-j}f(x+jh).\label{eqfwd2}
	\end{equation}
	Let $\sequ{\widetilde{V}}{j}$ be a sequence of independent copies of a random variable $\widetilde{V}$ uniformly distributed on $[0,1]$.  It turns out (cf. \cite{adlebi}) that
	\begin{equation}
		\Delta_h^kf(x)=h^k\widetilde{E}f^{(k)}\left(x+h(\widetilde{V}_1+\dots+\widetilde{V}_k)\right), \quad f\in C^k[0,1], \quad x\in[0,1].\label{eqfwsm}
	\end{equation}
	On the other hand, 	the usual first modulus of continuity of $f\in C[0,1]$ is defined as
	\begin{equation}\omega(f;\delta)=\sup\left\{\left|f(x)-f(y)\right|:\ x,y\in [0,1],\ |x-y|\leq\delta\right\},\quad \delta\geq 0,\label{dffimo}\end{equation}
	whereas the Ditzian-Totik second modulus of smoothness of $f$ with weight function $\phi\geq 0$ is defined as 
	\begin{equation}\omega_2^{\phi}(f;\delta)=\sup\left\{|\Delta_{h\phi(x)}^2f(x)|, \ x\pm h\phi(x)\in [0,1],\ 0\leq h\leq \delta\right\},\quad \delta \geq 0. \label{wesecm}\end{equation}
	If $\phi\equiv 1$, $\omega_2(f;\cdotp):=\omega_2^1(f;\cdotp)$ is the usual second modulus of continuity of $f$.  With regard to the Bernstein polynomials, we take from  \cs{P\u{a}lt\u{a}nea \cite{paonso} and \cite[pp. 94-96]{paappr}} the following approximation result.
	
	\begin{thm} Let $f\in C[0,1]$.  Then,
		\begin{equation}
			\left|B_m(f;x)-f(x)\right|\leq \cs{\frac{3}{2}}\omega_2\left(f;\frac{\sigma(x)}{\sqrt{m}}\right),\quad x\in[0,1], \label{beupbo}
		\end{equation}
		and 
		\begin{equation}
			\left\|B_m(f;x)-f(x)\right\|\leq \frac{5}{2}\omega_2^{\sigma}\left(f;\frac{1}{\sqrt{m}}\right).
			\label{beunbo}
		\end{equation}\label{tecobe}
	\end{thm}
	Let $k\in \mathbb{N}_0$ and $m>k$. Let $\widetilde{T}_k$ be a random variable taking values in $[0,k]$.  We define the positive linear operators
	\begin{equation}
		L_{m,k}(f;x)=\widetilde{E}f\left(\frac{\widetilde{S}_{m-k}(x)+\widetilde{T}_k}{m}\right),\quad f\in {\cal M}([0,1]), \quad x\in[0,1],\label{dfkant}
	\end{equation}
	\cs{where the random variable $\widetilde{S}_{m-k}(x)$ is defined in (\ref{empin})}.  In the particular case in which $\widetilde{T}_k=\widetilde{V}_1+\dots+\widetilde{V}_k$, ($\widetilde{T}_0=0$), where the sequence of random variables $\sequ{\widetilde{V}}{j}$ is defined as in (\ref{eqfwsm}), the operator $L_{m,k}$ is called the Bernstein-Kantorovich operator (see, for instance \cite{acrabe}, \cite{adcaon}, and the references therein). In such a case, we denote $B_{m,k}:=L_{m,k}$.  Note that $B_{m,0}=B_m$.
	
	The Bernstein-Kantorovich operators describe the derivatives of Bernstein polynomials when acting on smooth functions.  More precisely, we have (cf. \cite{adcast}) 
	\begin{equation}
		B_m^{(k)}(f;x)=(m)_k\widetilde{E}\Delta^k_{1/m}f\left(\frac{\widetilde{S}_{m-k}(x)}{m}\right),\ \quad f\in {\cal M}([0,1]),\quad x \in [0,1], \label{debedi}
	\end{equation}
	where $(m)_k=m(m-1)\dots(m-k+1), \ k\in \mathbb{N}$, $((m)_0=1)$ is the falling factorial, as well as
	\begin{equation}
		B_m^{(k)}(f;x)=\frac{(m)_k}{m^k}B_{m,k}(f^{(k)};x),\ \quad f\in C^k[0,1],\quad x \in [0,1]. \label{debeka}
	\end{equation}
	The approximation properties of the Bernstein-Kantorovich type operators defined in (\ref{dfkant}) are given in the following result.
	\begin{thm} \label{thkabo}
		Let $k\in \mathbb{N}_0$ and $m>k$. For any $f\in C[0,1]$, we have
		\begin{equation}
			|L_{m,k}(f;x)-f(x)|\leq 2k\omega\left(f;\frac{1}{m}\right)+\frac{3}{2}\omega_2\left(f;\frac{\sigma(x)}{\sqrt{m}}\right),\quad x\in[0,1], \label{thkab1}
		\end{equation}
		and
		\begin{equation}
			\|L_{m,k}(f;x)-f(x)\|\leq 2k\omega\left(f;\frac{1}{m}\right)+\frac{5}{2}\omega_2^{\sigma}\left(f;\frac{1}{\sqrt{m}}\right). \label{thkab2}
		\end{equation}
	\end{thm}
	\begin{proof}
		Consider the function $f_a(y)=f(ay), \ y\in[0,1]$, $0<a\leq 1$. It turns out that
		\begin{equation}
			\omega_2(f_a;\delta)\leq \omega_2(f;a\delta),\quad \omega_2^{\sigma}(f_a;\delta)\leq \omega_2^{\sigma}(f;\delta\sqrt{a}), \quad \delta \geq 0. \label{desmod}
		\end{equation}
		Actually, the first inequality readily follows from definition (\ref{wesecm}), whereas the second one also follows from (\ref{wesecm}) and the fact that
		\[\frac{a\sigma(x)}{\sigma(ax)}=\sqrt{\frac{a(1-x)}{1-ax}}\leq \sqrt{a}.\]
		Let $a=1-\frac{k}{m}$. By (\ref{dfkant}), we can write
		\begin{align*}&L_{m,k}(f;x)-f(x)=\widetilde{E}f\left(a\frac{\widetilde{S}_{m-k}(x)}{m-k}+\frac{\widetilde{T}_k}{m}\right)-
			\widetilde{E}f\left(a\frac{\widetilde{S}_{m-k}(x)}{m-k}\right)\\ &+	\widetilde{E}f\left(a\frac{\widetilde{S}_{m-k}(x)}{m-k}\right)-f\left(a x\right)+f\left( a x\right)-f(x).\end{align*}
		By (\ref{dffimo}), this implies that
		\[|L_{m,k}(f;x)-f(x)|\leq 2\omega\left(f;\frac{k}{m}\right)+|B_{m-k}(f_a;x)-f_a(x)|.\]
		Therefore, (\ref{thkab1}) follows from Theorem \ref{tecobe}, (\ref{desmod}), and the subadditivity property of $\omega\left(f;\cdotp\right)$. Inequality (\ref{thkab2}) is shown in a similar way.
	\end{proof}
	The rate of uniform convergence in (\ref{thkab2}) for $\|L_{m,k}(f;x)-f(x)\|$ cannot be better than $m^{-1}$ (see \cite{salowe}, \cite{adcast}, and the references therein).  Such a rate is attained  when $f\in C^2[0,1]$, since
	\[\omega_2^{\sigma}\left(f;\delta\right)\sim \|\sigma^2 f^{(2)}\|\delta^2,\quad \delta \rightarrow 0.\]
	Similarly, the order of magnitude of the upper bound in (\ref{thkab1}) is $m^{-1}$ at most.  However, this order of magnitude can be improved if $f\in C^k[0,1], \ k>2$, and some further assumptions on the derivatives of $f$ are made, as follows by considering the Taylor's expansion of $f$ (cf. \cite{adcaon}).  Even more, if $f$ is locally constant, the rates of convergence become exponential.  Indeed, for a fixed $x\in (0,1)$, consider the nonnegative convex function
	\[r(x,\theta)=\theta \log\frac{\theta}{x}+(1-\theta) \log\frac{1-\theta}{1-x}
	,\quad 0\leq \theta \leq 1.\]
	We take from \cite{adcaon} the following result.
	\begin{thm} In the setting of Theorem \ref{thkabo}, assume further that $f(t)=c$, $t\in(a,b)$, $0<a<b<1$, for some constant $c$.  For any $x\in \left(ma/(m-k), \cs{(mb-k)/(m-k)}\right)$, we have
		\begin{align*}&|L_{m,k}(f;x)-f(x)|\leq \\&\|f- c\|\left(\exp\left(-(m-k)r\kern-2pt \left(x,\frac{ma}{m-k}\right)\right)+\exp\left(-(m-k)r \kern-2pt \left(x,\cs{\frac{mb-k}{m-k}}\right)\right)\right).
		\end{align*}
		\label{thnull}
	\end{thm}
	\section{Confidence bands} \label{seconf}
	Unless otherwise specified, we assume from now on that $x\in (0,1)$ and that $F\in C[0,1]$ is the distribution function of a certain random variable $X$. We start from identity (\ref{decomp}) for the Bernstein polynomials when $\mu=F$, that is,
	\begin{equation}
		B_m(Y_n;x)-F(x)=B_m( Y_n-F;x)+B_m(F;x)-F(x),\label{decobe}
	\end{equation}
	where $\mathbb{Y}_n$ is the empirical process defined in (\ref{pampiri}).  Differentiating this formula, we obtain from (\ref{debeka})
	\begin{equation}
		B_m^{(k)}(Y_n;x)-F^{(k)}(x)=B_m^{(k)}( Y_n-F;x)+\frac{(m)_k}{m^k}B_{m,k}(F^{(k)};x)-F^{(k)}(x),\label{decobd}
	\end{equation}
	whenever $F\in C^k[0,1]$, $k\in \mathbb{N}_0$, and $m>k$.
	
	To give confidence intervals or bands, we need to estimate the tail probabilities of $B_m^{(k)}( Y_n-F;x)$ in both (\ref{decobe}) and (\ref{decobd}). In the case $k=0$, we see from (\ref{rabern}) that this random part is given by 
	\begin{equation}
		B_m(Y_n-F:x)=\frac{1}{n}\widetilde{E}\left(S_n (F(\widetilde{S}_m(x)/m))-nF(\widetilde{S}_m(x)/m)\right).  \label{berand}
	\end{equation}	
	Suppose that $k\in \mathbb{N}$ and $m>k$.  To deal with the derivatives of $F$, we see from (\ref{pampiri}) that
	\[	\Delta^1_{1/m}Y_n(x)=\frac{1}{n}\sum_{r=1}^{n}1_{(x, x+1/m]}(X_r), \] 
	thus implying, by virtue of (\ref{eqfwdi}) and  (\ref{eqfwd2}), that
	\[	\Delta^k_{1/m}Y_n(x)=	\Delta^{k-1}_{1/m}\left(\Delta^1_{1/m}Y_n\right)(x)=\frac{1}{n}\sum_{j=0}^{k-1}{k-1 \choose j}(-1)^{k-1-j}\sum_{r=1}^{n}1_{\left(x+\frac{j}{m}, x+\frac{j+1}{m}\right]}(X_r). \]
	We thus have from (\ref{debedi}) that the estimator of $F^{(k)}(x)$ is given by
	\begin{align}
		&B_m^{(k)}(Y_n;x)=(m)_k\widetilde{E}\Delta^k_{1/m}Y_n\left(\frac{\widetilde{S}_{m-k}(x)}{m}\right)\nonumber\\&= \frac{(m)_k}{n}\sum_{l=0}^{m-k}p_{m-k,l}(x)\sum_{j=0}^{k-1}{k-1 \choose j}(-1)^{k-1-j}\sum_{r=1}^{n}1_{\left(\frac{l+j}{m}, \frac{l+j+1}{m}\right]}(X_r). \label{derbed}
	\end{align}
	We point out that, for $k=1$, this estimator was considered by many authors (see, for instance, Vitale \cite{viaber}, Babu et al. \cite{baappl}, Bouezmarni and Rollin \cite{borobe},  Ouimet \cite{ouassy}, and Wang and Lu \cite{waluap}, among others).  On the other hand, using representation (\ref{eqlaym}), the random term on the right-hand side in (\ref{decobd}) can be written as
	\begin{align}
		&	\frac{1}{(m)_k}B_m^{(k)}(Y_n-F;x)=\widetilde{E}\Delta^k_{1/m}\left(Y_n-F\right)\left(\frac{\widetilde{S}_{m-k}(x)}{m}\right)\nonumber\\&= \frac{1}{n}\sum_{j=0}^{k-1}{k-1 \choose j}(-1)^{k-1-j} \widetilde{E}\left(S_n\left(F\left(\frac{\widetilde{S}_{m-k}(x)+j+1}{m}\right)\right)-S_n\left(F\left(\frac{\widetilde{S}_{m-k}(x)+j}{m}\right)\right)\right.\nonumber\\&\phantom{\frac{1}{n}\sum_{j=0}^{k-1}{k-1 \choose j}{k-1 \choose j}(-1)^{k-1-j} }\left.-n\left(F\left(\frac{\widetilde{S}_{m-k}(x)+j+1}{m}\right)-F\left(\frac{\widetilde{S}_{m-k}(x)+j}{m}\right)\right)\right).\label{derapp}
	\end{align}
	
	To give confidence bands for $F^{(k)}$, we use the celebrated D-K-W inequality (cf. \cite{dvkias}) in the final form shown by Massart \cite{mathet}, which states that
	\begin{equation}	P\left(\left\|\frac{S_n(x)}{n}-x\right\|> \delta\right)\leq  2e^{-2n\delta^2}, \quad \delta>0. \label{expbmg}\end{equation}
	Our first main result is the following. 
	\cs{\begin{thm} Let $k\in \mathbb{N}_0$.  Assume that $F\in C^{k}[0,1]$.  Then,
		\begin{equation}
			P\left(\left\|B_m^{(k)}(Y_n;x)-F^{(k)}(x)\right\|> \frac{l_{k,\alpha}(m)}{\sqrt{n}}\right)  \leq \alpha, \label{eqbth1}
		\end{equation}
		where $m>\max (k,1)$ is the first positive integer such that 
		\begin{equation}
			\left\|\frac{(m)_k}{m^k} B_{m,k}(F^{(k)};x)-F^{(k)}(x)\right\|\leq \frac{1}{2}l_{k,\alpha}(m)\frac{1}{\sqrt{n}},\label{cowilk} 
		\end{equation}
		and
		\begin{equation}
			l_{k,\alpha}(m)=2^k(m)_k \sqrt{2 \log(2/\alpha)}.\label{dflkam}
		\end{equation}
		\label{thcoba}\end{thm}
		
		We see from (\ref{eqbth1}) and (\ref{dflkam}) that the length of the confidence band increases as $k$ increases.  As a counterpart, the degree $m$ of the Bernstein estimator decreases as $k$ increases.  This will be clear after the comments following Corollaries \ref{coubof} and \ref{coubod} below.
	
Setting $k=0,1$ in Theorem \ref{thcoba}, we obtain confidence bands for $F$ and $F^{(1)}$, respectively.
\begin{cor}
	If $F\in C[0,1]$, then
	\begin{equation}
			P\left(\left\|B_m(Y_n;x)-F(x)\right\|> \frac{\sqrt{2 \log(2/\alpha)}}{\sqrt{n}}\right)  \leq \alpha, \label{eqbco1}\end{equation}
			where $m>1$ is the fist positive integer such that
					\begin{equation}
				\left\| B_{m}(F;x)-F(x)\right\|\leq \frac{\sqrt{2 \log(2/\alpha)}}{2\sqrt{n}}.\label{cowil0} 
			\end{equation}
\label{coubof}	
\end{cor} 
Suppose that the Ditzian-Totik modulus of $F$ fulfills the Lipschitz condition
\[\omega_2^{\sigma}(F;h)\leq Ch^{\beta},\quad h\geq 0,\quad 0<\beta\leq 2,\]
for some constant $C>0$. By (\ref{beunbo}), condition (\ref{cowil0}) is satisfied if 
\[\frac{5C}{2}\frac{1}{m^{\beta/2}}\leq \frac{\sqrt{2 \log(2/\alpha)}}{2\sqrt{n}}.\]
In other words, the degree $m$ of the Bernstein polynomial has the order of magnitude $m\sim n^{1/\beta}$, thus heavily depending on the degree of smoothness of $F$.
\begin{cor}
	If $F\in C^1[0,1]$, then
	\begin{equation}
		P\left(\left\|B_m^{(1)}(Y_n;x)-F^{(1)}(x)\right\|> \frac{2m\sqrt{2 \log(2/\alpha)}}{\sqrt{n}}\right)  \leq \alpha, \label{eqbco2}\end{equation}
	where $m>1$ is the fist positive integer such that
	\begin{equation}
		\left\| B_{m,1}(F^{(1)};x)-F^{(1)}(x)\right\|\leq \frac{m\sqrt{2 \log(2/\alpha)}}{\sqrt{n}}.\label{cowil1} 
	\end{equation}
	\label{coubod}	
\end{cor}
In view of (\ref{thkab1}), suppose that \[	\left\| B_{m,1}(F^{(1)};x)-F^{(1)}(x)\right\|\leq \frac{C}{m^{\beta/2}},\quad 0<\beta\leq 2,\]
for some constant $C>0$. Then, condition (\ref{cowil1}) is satisfied if 
\[\frac{C}{m^{\beta/2}}\leq \frac{m\sqrt{2 \log(2/\alpha)}}{\sqrt{n}}.\]
This implies that the degree $m$ of the Bernstein polynomial has the order of magnitude $m\sim n^{1/(\beta+2)}$. However, the length $l$ of the confidence band has the order of magnitude $l\sim n^{-\beta/(2\beta+4)}$, as follows from (\ref{eqbco2}). As in Corollary \ref{coubof}, this length $l$ strongly depends on the degree of smoothness of $F^{(1)}$.

}

Finally, assume that $F(x)=x, \ x\in[0,1]$. Then, we have from (\ref{berpol}) and (\ref{debeka})
	\begin{equation}
		\left\|\frac{(m)_k}{m^k} B_{m,k}(F^{(k)};x)-F^{(k)}(x)\right\|=0,\quad k=0,1,\quad m\geq 2. \label{center}
	\end{equation}
	In other words, condition (\ref{cowilk}) in Theorem \ref{thcoba} holds for $m\geq 2$.  This leads us to state the following simple result for the uniform distribution.
	\begin{cor} If $F(x)=x, \ 0\leq x\leq 1$, then
		\begin{equation}P\left(\left\|B_2(Y_n;x)-x\right\|>\delta\right)\leq 2 e
			^{-8n\delta^2},\quad \delta>0.\label{unidis}\end{equation}
		and
		\begin{equation}P\left(\left\|B_2^{(1)}(Y_n;x)-1\right\|>\delta\right)\leq 2 e^{- n\delta^2/8},\quad \delta>0.\label{unider}\end{equation}
		\label{counba}	
	\end{cor}
	
	By (\ref{eqlaym}) and (\ref{rabern}), the estimator
	\[B_2(Y_n;x)=2\sigma^2(x)Y_n(1/2)+x^2=2\sigma^2(x)\frac{S_n(1/2)}{n}+x^2\]
	is a random polynomial of degree $2$ which only involves the random variable $S_n(1/2)$ having the binomial distribution with parameters $n$ and $1/2$.  To estimate the uniform distribution, this estimator is much simpler and much more efficient than the uniform empirical process used in the D-K-W inequality (compare inequalities (\ref{expbmg}) and  (\ref{unidis})).  In addition, the  
	MSE and MISE of the estimator $B_2(Y_n;x)$ are easily computed, since
	\[MSE= 4\sigma^4(x) Var(Y_n(1/2))=\frac{\sigma^4(x)}{n}, \quad MISE=\frac{1}{30n}. \]
	
	\section{Confidence intervals for $F$} \label{secoin}
	In this section, we give confidence intervals for $F(x)$, for any $x\in (0,1)$.  For the sake of simplicity, we assume from now on that $0<F(x)	<1, \ x\in(0,1)$. 
	
	To this end, we consider the nonnegative, increasing, and convex function
	\begin{equation}
		\tau(\epsilon)=(1+\epsilon )\log(1+\epsilon)-\epsilon, \ \epsilon\geq 0; \quad \tau(\epsilon)\geq \frac{\epsilon^2}{2}-\frac{\epsilon^3}{6}\geq \frac{\epsilon^2}{3},\quad 0\leq \epsilon\leq 1.  \label{dftaue}
	\end{equation} 	
		We state the basic result of this section.
	\cs{\begin{thm} Let $x\in (0,1)$ and $\epsilon>0$. Then,
		\begin{align}
			&	P\left(\left|B_m(Y_n;x)-F(x)\right|\geq \epsilon B_m(\sigma^2(F);x)+ \left|B_m(F;x)-F(x)\right| \right) \nonumber \\&\leq 2 exp\left(-\frac{n }{q_m(x)}B_m\left(\sigma^2 (F);x\right)\tau(\epsilon)\right), \label{eqbth2}	\end{align} 
				where
\begin{equation}q_m(x)=1-x^m-(1-x)^m.\label{dfqmx}		
	\end{equation}\label{thcob0}\end{thm}}

As a first application, we consider the case of the uniform distribution, which has interest in itself.
\begin{cor} Let $F(x)=x, \ 0\leq x\leq 1$, and $\delta>0$. Then,
	\begin{equation}
		P\left(\left|B_2(Y_n;x)-x\right|\geq 4\sigma^2(x)\delta\right)\leq 2 e^{-n\tau(8\delta)/4 }\leq 2 exp\left(-8n\delta^2\left(1-\frac{8}{3}\delta\right)\right) \label{counu1},
	\end{equation}
	where the last inequality holds whenever $\delta\leq 1/8$. \label{counup}
\end{cor}

	Observe that the upper bound in (\ref{counu1}) is uniform in $x\in (0,1)$.  This means that the same confidence level works for any $x\in(0,1)$.  On the other hand, the length of the confidence interval in (\ref{counu1}) is shorter than that of the confidence band in Corollary \ref{counba}, except for $x=1/2$, case in which both lengths are the same.  In fact, near the endpoints of $[0,1]$, the length of the confidence interval is very short.
	
	\cs{Let
		\begin{equation}l_\alpha=\sqrt{3\log(2e^{1/3}/\alpha)}. \label{dflalph}\end{equation}
	For a general distribution function $F$, we give the following result.
	\begin{cor}
		Let $x\in(0,1)$.  Assume that $n\geq l_\alpha^2/\sigma^2(F(x))$. Then,
		\begin{equation}
			P\left(\left|B_m(Y_n;x)-F(x)\right|\geq \left(2 l_\alpha+\frac{1}{l_\alpha} \right)\frac{ \sigma(F(x))}{\sqrt{n}} \right) \leq\alpha, \label{intek0}
		\end{equation}
		where $m>1$ is the first positive integer such that
		\begin{equation}
			\left|B_m(\sigma^2(F);x)-\sigma^2(F(x))\right|\leq\frac{\sigma^2(F(x))}{l_\alpha^2}, \label{ink0c1}
		\end{equation}
		and 
		\begin{equation}
			\left|B_m(F;x)-F(x)\right|\leq \frac{\sigma\left(F(x)\right)l_\alpha}{\sqrt{n}}. \label{ink0c2}
		\end{equation} \label{coink0}
	\end{cor}}
	
\cs{Comparing this result with Corollary \ref{coubof}, we see that the length of the confidence intervals and bands for $F$ have the same order of magnitude $1/\sqrt{n}$.  On the other hand, assume that $F$ and $\sigma^2(F)$ fulfil the Lipschitz condition
	\begin{equation}
		\max\left(\omega_2(F;h),\omega_2(\sigma^2\left(F\right); h) \right)\leq C h^{\beta}, \quad 0<\beta \leq 2,\label{comaxFs}
	\end{equation}
	for some constant $C>0$. By (\ref{beupbo}), conditions (\ref{ink0c1}) and (\ref{ink0c2}) are satisfied if 
	\begin{equation}
		\frac{3}{2}C\frac{\sigma^{\beta}(x)}{m^{\beta/2}}\leq \frac{\sigma^2(F(x))}{l_\alpha^2}\quad \hbox{and}\quad 	\frac{3}{2}C\frac{\sigma^{\beta}(x)}{m^{\beta/2}}\leq  \frac{\sigma\left(F(x)\right)l_\alpha}{\sqrt{n}}, \label{verink}
	\end{equation}
	respectively.  This means that the order of magnitude of the degree $m$ of the Bernstein polynomial is $m\sim n^{1/\beta}$, again heavily depending on the degree of smoothness of $F$.  In this respect, suppose that $F(x)=c$, $x\in (a,b)$, $0<a<b<1$.  For $x\in (a,b)$, Theorem (\ref{thnull}) implies that condition (\ref{ink0c2}) is satisfied if 
	
	\[\|F- c\|\left(e^{-mr(x,a)}+e^{-mr(x,b)}\right)\leq \frac{\sigma\left(F(c)\right)l_\alpha}{\sqrt{n}},\]
	and therefore, the order of magnitude of $m$ is $m\sim\log n$. Similar considerations can be made with regard to condition (\ref{ink0c1}).
	
	No proper comparisons between expressions (\ref{intek0}) and (\ref{assint}) can be made, because the confidence interval in (\ref{intek0}) is explicit, whereas that in (\ref{assint}) is asymptotic. This notwithstanding, the length of the confidence interval in (\ref{intek0}) is a little worse than that in (\ref{assint}).  For instance, if $\alpha=0.05$, then
	\[2 l_\alpha+\frac{1}{l_\alpha}=7.2353,\quad z_{\alpha/2}=1.96, \]
	as follows from (\ref{dflalph}).  However, in both cases, the lengths of the confidence intervals have the same order of magnitude. On the other hand, formula  (\ref{intek0}) holds for any $F\in C[0,1]$, whereas (\ref{assint}) only holds for $F\in C^2[0,1]$.  In this last case, the size $m$ of the Bernstein polynomial has order of magnitude $n^{2/3}$ in (\ref{assint}), and $n^{1/2}$ in  (\ref{intek0}), as follows from  (\ref{comaxFs}) and  (\ref{verink}).}

	\section{Confidence intervals for $F^{(k)}$} \label{secoid}
	
In this section, we fix $k\in \mathbb{N}$ and assume that $m>k$.  We suppose that $F\in C^{k}[0,1]$ and denote by $\rho(x)=F^{(1)}(x)$. We also assume that $\rho(x)>0,\ x\in (0,1)$.  	

\cs{Denote by $\lfloor y \rfloor$ the integer part of $y\in \mathbb{R}$. Let
		\begin{equation}
			r_k(m)=\frac{md_k}{(m)_k {2(k-1) \choose k-1} },\label{dfrkm}
		\end{equation}
		where	
				\begin{equation}
d_k= {k-1 \choose \lfloor k/2\rfloor}. 	\label{leprz1}
		\end{equation}
		On the other hand, consider the following particular case of the positive linear operator defined in (\ref{dfkant})
		\begin{equation}
			L_{m,k}(\rho;x)=\widetilde{E}\rho\left(\frac{\widetilde{S}_{m-k}(x)+\widetilde{R}_{k-1}+\widetilde{V}}{m}\right),\label{dfkan2}
		\end{equation}
		where $\widetilde{V}$ is uniformly distributed on $[0,1]$ and 
				\begin{equation}
			\widetilde{P}\left(\widetilde{R}_{k-1}=j\right)=\frac{\displaystyle{k-1 \choose j}^2}{\displaystyle{2(k-1) \choose k-1}},\quad j=0,1,\dots, k-1. \label{leprz3}
		\end{equation} 
		The following result is analogous to Theorem \ref{thcob0}.
			\begin{thm} Let $x\in(0,1)$ and $\delta>0$.  Then,
			\begin{eqnarray}
			&&	P\left(\left|B_m^{(k)}(Y_n;x)-F^{(k)}(x)\right|\geq \delta L_{m,k}(\rho, x)+\left|\frac{(m)_k}{m^k}B_{m,k}(F^{(k)};x)-F^{(k)}(x)\right| \right)\nonumber \\ &&\leq 2 exp \left(-\frac{n}{(m)_kd_k r_k(m)}L_{m,k}(\rho;x)\tau(r_k(m)\delta)\right).\label{thdere}
			\end{eqnarray}
		 \label{thder}
			\end{thm}

The following result is the main result of this section. 
	\begin{thm} Let $x\in (0,1)$.  Set
		\begin{equation}
			s_k(m)=\frac{(m)_k}{\sqrt{m}}\sqrt{ {2(k-1) \choose k-1}}.	\label{dfskm}
		\end{equation}
		\label{thboud}
		Assume that 
		\begin{equation}
		\frac{m}{n}\leq \frac{\rho(x)}{ l_\alpha^2d_k^2 }{2(k-1) \choose k-1}, \label{thebou1}
		\end{equation}
		where $l_\alpha$ and $d_k$ are defined in (\ref{dflalph}) and (\ref{leprz1}), respectively. Then,
		\begin{equation}
			P\left(\left|B_m^{(k)}(Y_n;x)-F^{(k)}(x)\right|\geq \left(2l_\alpha+\frac{1}{l_\alpha} \right) \sqrt{\frac{\rho(x)}{n}}s_k(m) \right)\leq \alpha, \label{thbou2}
		\end{equation}
		where $m>1$ is the first postive integer such that
		\begin{equation}
			|L_{m,k}(\rho;x)-\rho(x)|\leq \frac{\rho(x)}{l_\alpha^2},\label{thcon1}
		\end{equation}
		and
		\begin{equation}
			\left|\frac{(m)_k}{m^k}B_{m,k}(F^{(k)};x)-F^{(k)}(x)\right|\leq l_\alpha\sqrt{\frac{\rho(x)}{n}}s_k(m). \label{thcon2}
		\end{equation}\label{thcodf}
	\end{thm}
	
	The comparison of this result with Theorem \ref{thcoba} reveals the following.  For $k\geq 1$, the length of the confidence interval for $F^{(k)}$ is that of the confidence band for $F^{(k)}$ reduced by a factor $1/\sqrt{m}$, provided that the constants are dismissed.  This follows from (\ref{dflkam}) and (\ref{dfskm}).

If $k=1$, then $d_1=1$ and $s_1(m)=\sqrt{m}$, as follows from (\ref{leprz1}) and (\ref{dfskm}), respectively. On the other hand, we see from (\ref{dfkan2}) and the comments following (\ref{dfkant}) that $L_{m,1}=B_{m,1}$.  Therefore, the following result on density estimation is an immediate consequence of Theorem \ref{thcodf}.
\begin{cor}
	Let $x\in (0,1)$.  Assume that 
	\[	\frac{m}{n}\leq \frac{\rho(x)}{ l_\alpha^2}.\]
	Then,
			\begin{equation}
		P\left(\left|B_m^{(1)}(Y_n;x)-\rho(x)\right|\geq \left(2l_\alpha+\frac{1}{l_\alpha} \right) \sqrt{\frac{m\rho(x)}{n}} \right)\leq \alpha, \label{c1bou2}
	\end{equation}
	where $m>1$ is the first positive integer such that
			\begin{equation}
		\left|B_{m,1}(\rho;x)-\rho(x)\right|\leq \min\left(\frac{\rho(x)}{l_\alpha^2},l_\alpha\sqrt{\frac{m\rho(x)}{n}}\right). \label{c1con2}
	\end{equation}\label{cocodf}
	
\end{cor}
In view of (\ref{thkab1}), suppose that
\[\left|B_{m,1}(\rho;x)-\rho(x)\right|\leq \frac{C(x)}{m^{\beta/2}},\quad 0<\beta\leq 2,\]
for some positive constant $C(x)$. Then, condition (\ref{c1con2}) is satisfied if
\[\frac{C(x)}{m^{\beta/2}}\leq \min\left(\frac{\rho(x)}{l_\alpha^2},l_\alpha\sqrt{\frac{m\rho(x)}{n}}\right).\]
	This implies that the degree $m$ of the Bernstein polynomial has the order of magnitude $m\sim n^{1/(\beta+1)}$. As a consequence, the length $l$ of the confidence interval in (\ref{c1bou2}) has the order of magnitude $l\sim n^{-\beta/(2\beta+2)}$.
	
	Recall that, under analogous smoothness assumptions on $\rho(x)$, the length $l$ of the confidence band for $\rho(x)$ has the order of magnitude $l\sim n^{-\beta/(2\beta+4)}$, as seen in the comments following Corollary \ref{coubod}.
	
	Finally, suppose that $\rho(x)=c>0,\ x\in (a,b)$, with $0<a<b<1$. Then, it follows from Theorem \ref{thnull} that
\[	\left|B_{m,1}(\rho;x)-\rho(x)\right|\leq \|\rho- c\|\left(\exp\left(-(m-1)r\kern-2pt \left(x,\frac{ma}{m-1}\right)\right)+\exp\left(-(m-1)r \kern-2pt \left(x,\frac{mb-1}{m-1}\right)\right)\right),
\]
whenever $x\in \left(ma/(m-1), (mb-1)/(m-1)\right)$. Therefore, for any $x$ in the previous interval, we can take $m$ such that $m+\log m\sim\log n$, as follows from (\ref{c1con2}). By (\ref{c1bou2}), this implies that the length of the confidence interval has the order of $\left(\log n/n\right)^{1/2}$.

\section{Simulations} \label{sesimu}
To illustrate Corollaries \ref{counba}, \ref{counup}, \ref{coink0}, and \ref{cocodf}, we consider the family of distribution functions 
\begin{equation}
	F_\beta(x)=x^{\beta},\quad 0\leq x\leq 1,\quad \beta>0.\label{dffbet}
\end{equation}
Firstly, consider the uniform distribution ($\beta=1$). Figure \ref{fig:graficos} (a), illustrates Corollary \ref{counba}, establishing a comparison between the confidence bands obtained for $B_2(Y_n;x)$ and those obtained using the D-K-W inequality and the empirical process as the estimator. Figure \ref{fig:graficos} (b) shows the confidence intervals obtained in Corollary \ref{counup}. 
\begin{figure}[htbp]
	\centering
	\begin{subfigure}[b]{0.495\textwidth}
		\includegraphics[width=\textwidth]{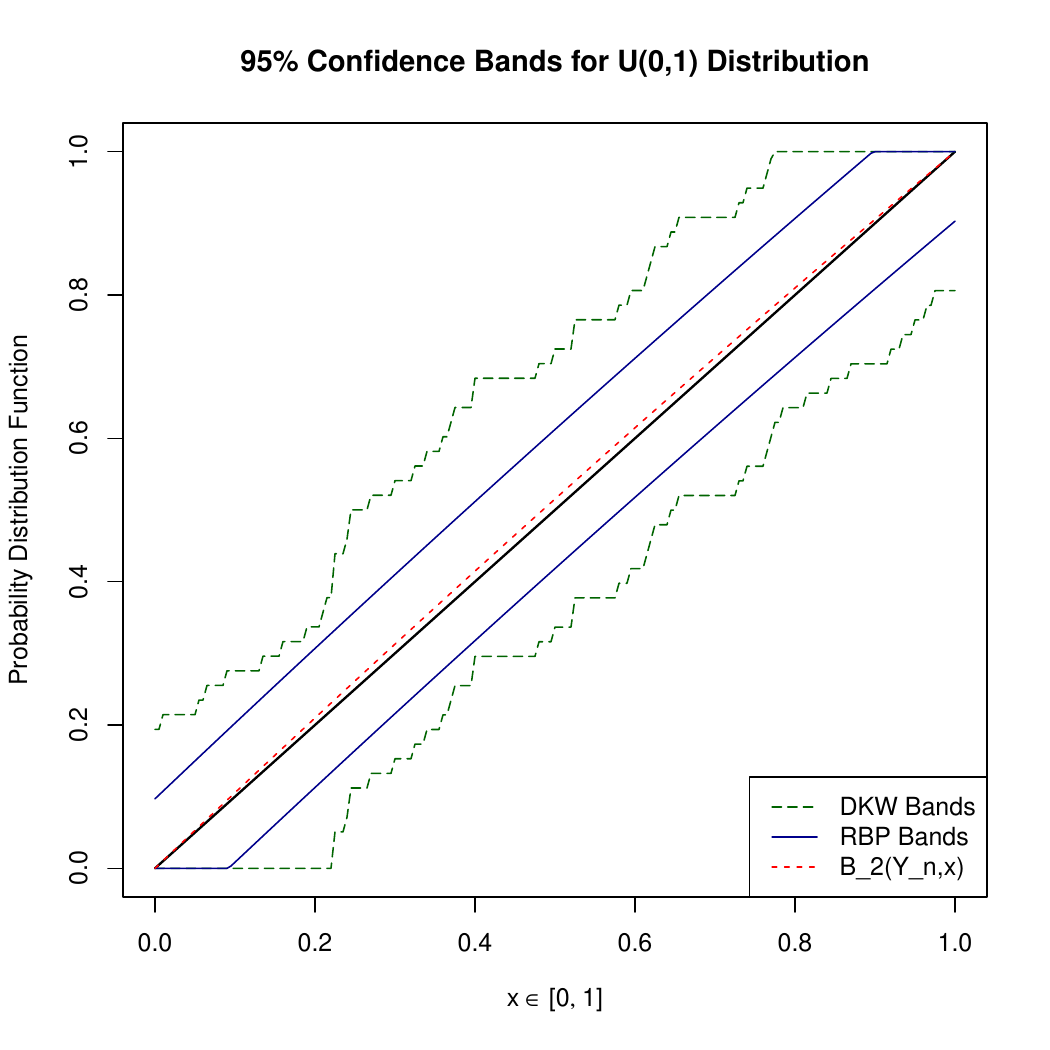}
		\caption{\small  $B_2(Y_n;x)$ confidence bands (blue continuous lines) vs. D-K-W (green dotted lines)}
	\end{subfigure}
	\hfill
	\begin{subfigure}[b]{0.495\textwidth}
		\includegraphics[width=\textwidth]{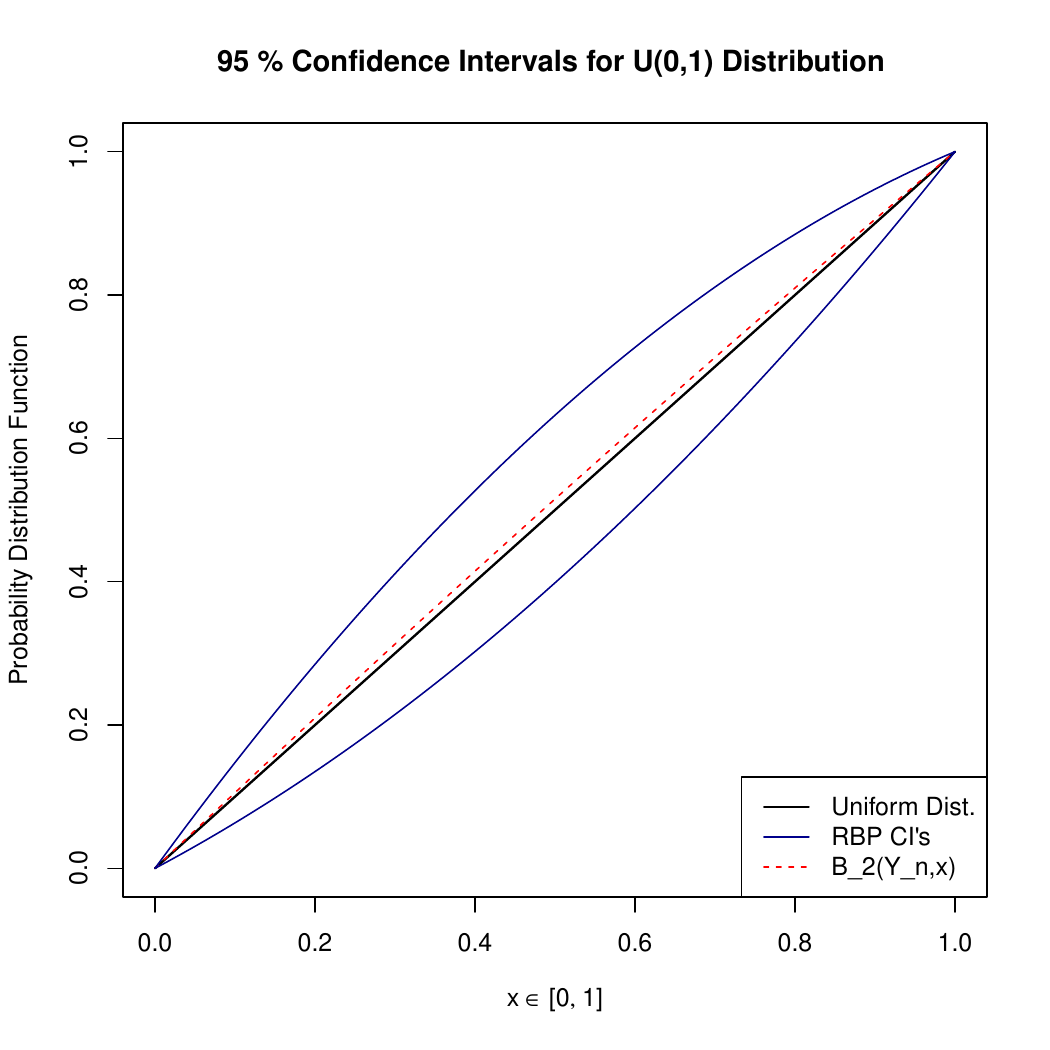}
		\caption{\small Confidence intervals based on $B_2(Y_n;x)$ (blue continuous lines)}
	\end{subfigure}
	\caption{\small $95\%$ confidence bands (a) and confidence intervals (b) based on $B_2(Y_n;x)$ for the uniform Distribution for a sample size $n=49$.}
	\label{fig:graficos}
\end{figure}

Now, we focus our attention on $F_{\beta}$ as in (\ref{dffbet}) and $\beta\neq 1$. We can estimate the first and second moduli of smoothness, as defined in (\ref{dffimo}) and (\ref{wesecm}), as follows.
\begin{lem} Let $0\leq h\leq 1/2$.  Then,
	\begin{equation}
		\omega(F_\beta;h)=h^\beta, \ 0<\beta\leq 1, \quad 	\omega(F_\beta;h)\leq \beta  h, \ 1<\beta, \label{befimo}
	\end{equation}
	and
		\begin{equation}
		\omega_2(F_\beta;h)=|2^{\beta}-2|h^\beta, \ 0<\beta\leq 2, \quad 	\omega_2(F_\beta;h)\leq \beta (\beta-1) h^2, \ 2<\beta. \label{besemo}
	\end{equation}
	\label{lebemo}
		\end{lem}
		In first place, we discuss the order of magnitude of the degree $m$ of the Bernstein polynomial in the setting of Corollary \ref{coink0}, when we consider any of the distribution functions $F_\beta,\ \beta\neq 1$, defined in (\ref{dffbet}).  We claim that
		\begin{equation}
			m\sim n^{1/\beta},\quad \beta \in (0,1)\cup (1,2), \qquad
			m\sim n^{1/2},\quad \beta\geq 2. \label{conmbe}
		\end{equation}
		Actually, as far as the order of magnitude of $m$ is concerned, condition (\ref{ink0c2}) is stronger than (\ref{ink0c1}).  Assume that $\beta \in (0,1)\cup (1,2)$. By (\ref{beupbo}) and (\ref{besemo}), we have
		\[\left|B_{m}(F_\beta;x)-F_\beta(x)\right|\leq \frac{3}{2}|2^\beta-2|\frac{(\sigma(x))^\beta}{m^{\beta/2}}.\]
		Thus, condition  (\ref{ink0c2}) is satisfied if 
		\[\frac{3}{2}|2^\beta-2|\frac{(1-x)^{\beta/2}}{m^{\beta/2}}\leq \frac{(1-x^\beta)^{1/2}l_\alpha}{\sqrt{n}}.\]
		This shows the first statement in (\ref{conmbe}).  The second one is shown in a similar way.
		
	If $\beta\geq 2$, then $F_\beta\in C^{2}[0,1]$. Expression (\ref{assint}) implies that we must choose $m\sim n^{2/3}$, according to the criterion of minimizing the MSE or MISE, whereas we can choose $m\sim n^{1/2}$, as shown in (\ref{conmbe}).  Simulations for $\beta=3/2$ ($F_\beta \not \in C^{2}[0,1]$) and $\beta=2$ are provided in Figure \ref{fig:grafico2}. 
	
		\begin{figure}[htbp]
				\centering
				\begin{subfigure}[b]{0.495\textwidth}
					\includegraphics[width=\textwidth]{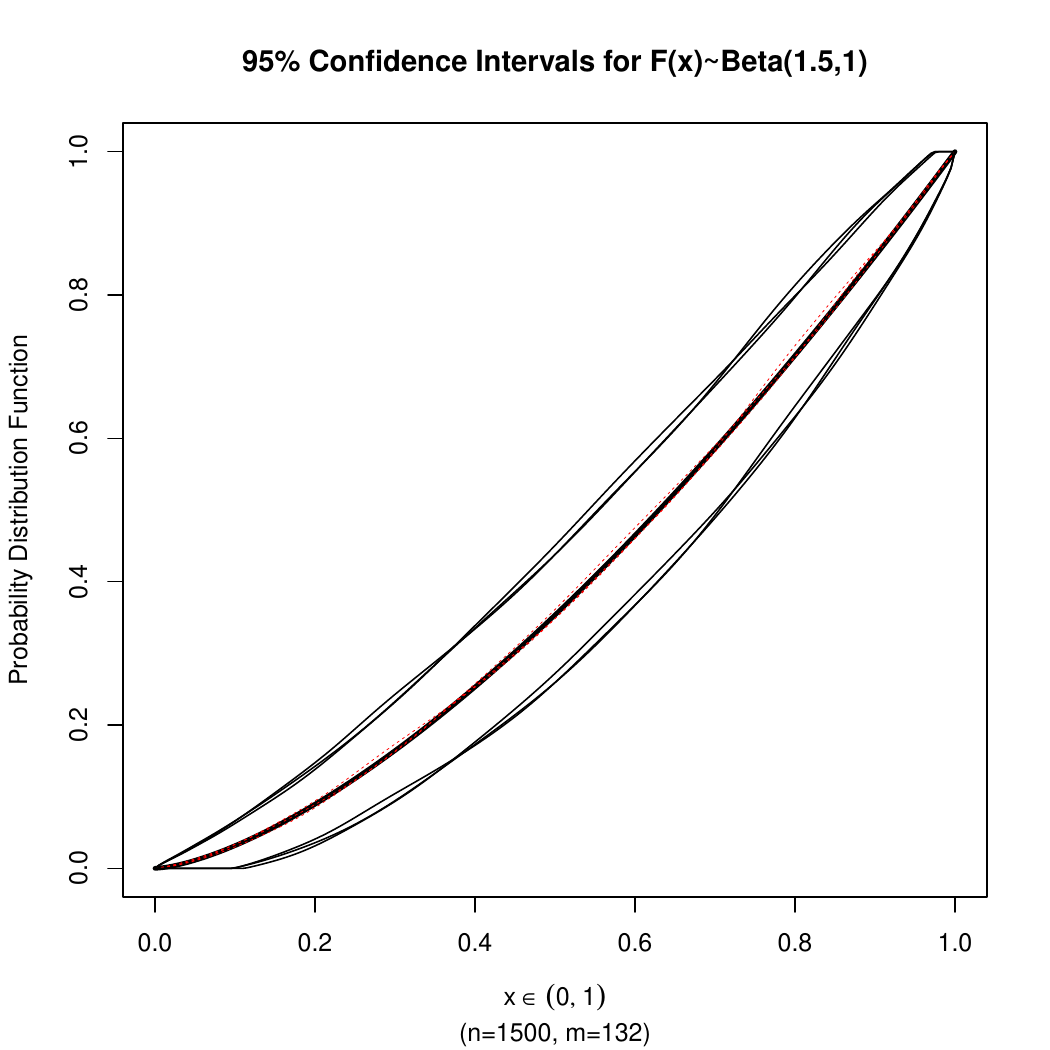}
					\caption{\small $F_\beta(x), \ \beta=3/2$}
					
				\end{subfigure}
				\hfill
				\begin{subfigure}[b]{0.495\textwidth}
					\includegraphics[width=\textwidth]{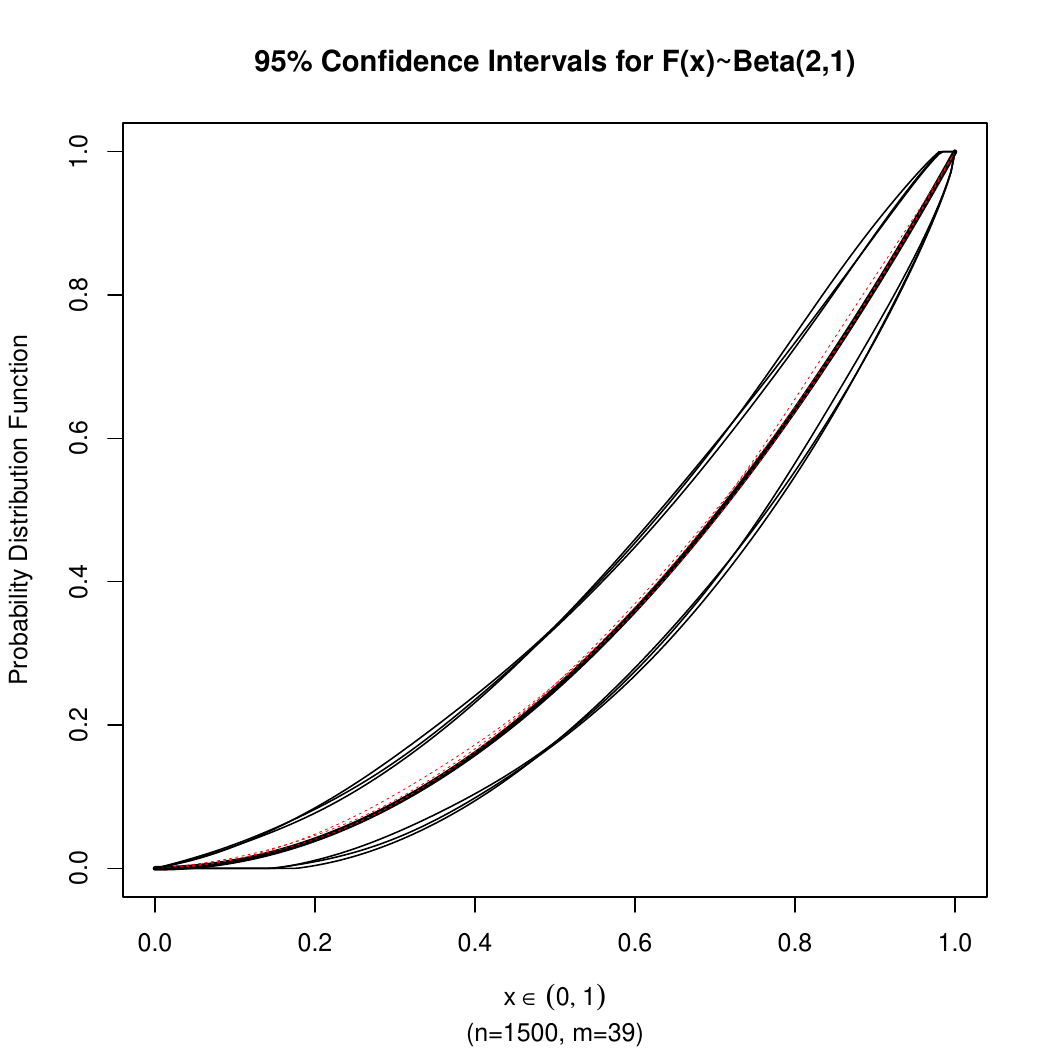}
					\caption{\small $F_\beta(x), \ \beta= 2$}
					
				\end{subfigure}
				\caption{\small $95\%$ confidence intervals for $F_\beta(x)$ with 3 replicates in each case and sample size $n=1500$. Optimal degree $m\sim n^{1/\beta}$ selected according with the smoothness condition. In red dotted line, the corresponding Bernstein polynomial estimator, $B_m(Y_n; x)$, for $F_\beta(x)$.}
				\label{fig:grafico2}
			\end{figure}

		In second place, consider the same question in the setting of Corollary \ref{cocodf} with respect to any of the probability density functions $\rho_\beta(x)=\beta x^{\beta-1},\ \beta>1$.  We claim that 
				\begin{equation}
			m\sim n^{1/\beta},\quad 1<\beta<3 , \qquad
			m\sim n^{1/3},\quad \beta\geq 3. \label{conmbed}
		\end{equation}
		We only show the case $ 1<\beta<2$. From (\ref{thkab1}) and Lemma \ref{lebemo}, we get
		\[		
			|B_{m,1}(\rho_\beta;x)-\rho_\beta(x)|\leq 2\omega\left(\rho_\beta;\frac{1}{m}\right)+\frac{3}{2}\omega_2\left(\rho_\beta;\frac{\sigma(x)}{\sqrt{m}}\right)\leq  \frac{2 \beta}{m^{\beta-1}}+\frac{3}{2} \left(2-2^{\beta-1}\right)\frac{(\sigma(x))^{\beta-1}}{m^{(\beta-1)/2}}.\label{thkab}
	\]
		Therefore, condition (\ref{c1con2}) is satisfied if $m\sim n^{1/\beta}$. Simulation results for $\beta=5/2$ ($\rho_\beta \not \in C^{2}[0,1]$) and $\beta=3$ are shown in Figure \ref{fig:grafico3}.

		\begin{figure}[htbp]
			\centering
			\begin{subfigure}[b]{0.495\textwidth}
				\includegraphics[width=\textwidth]{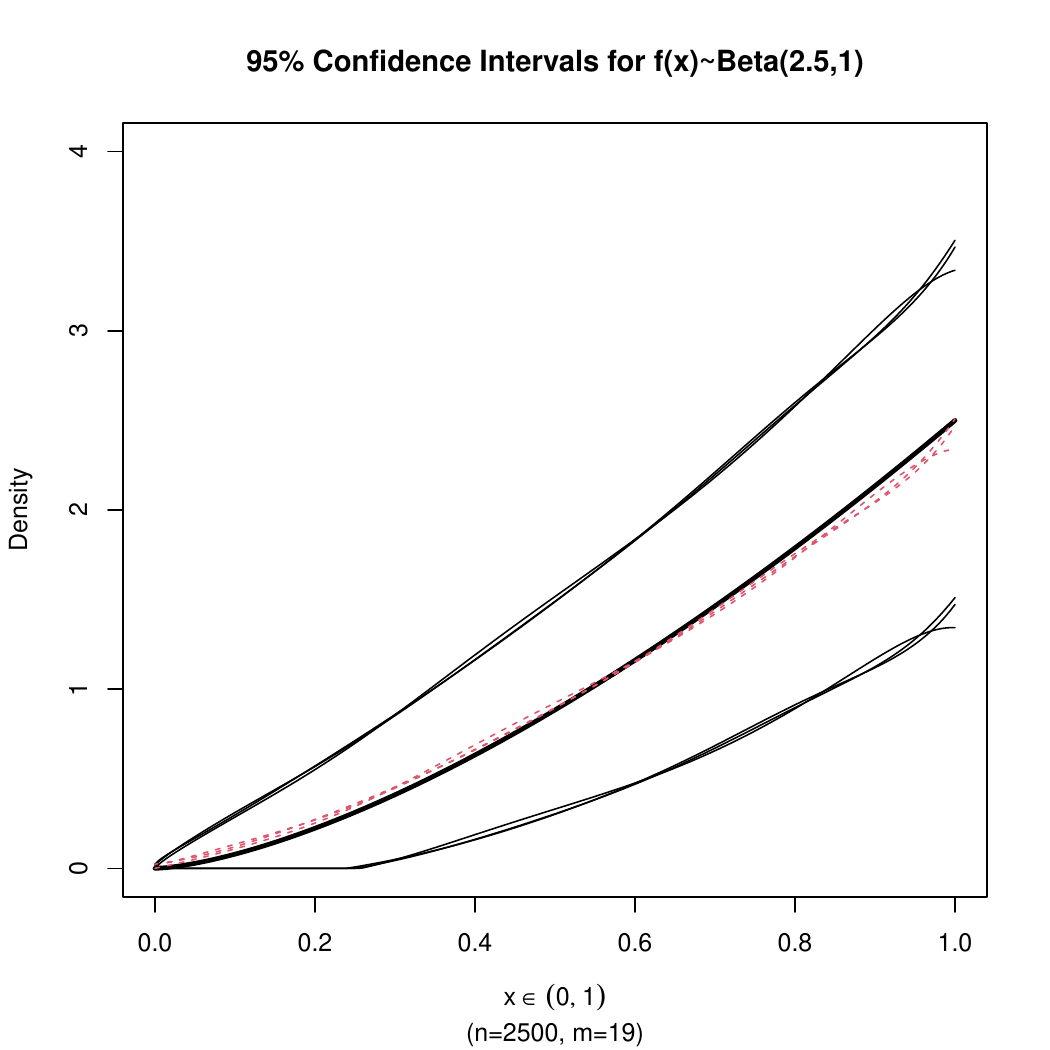}
				\caption{\small $\rho_\beta(x), \beta= 5/2$}
				
			\end{subfigure}
			\hfill
			\begin{subfigure}[b]{0.495\textwidth}
				\includegraphics[width=\textwidth]{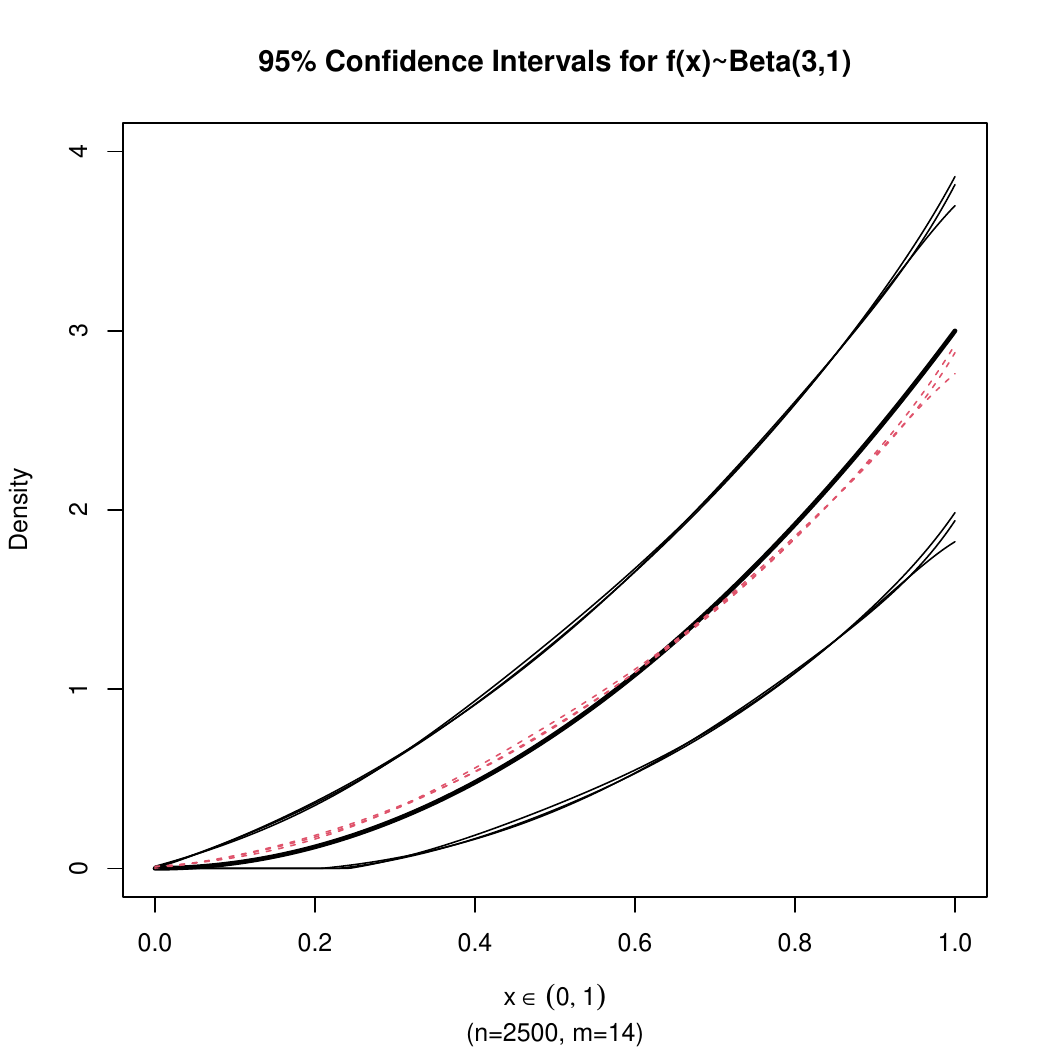}
				\caption{\small $\rho_\beta(x), \beta= 3$}
			\end{subfigure}
			\caption{\small $95\%$ confidence intervals for $\rho_\beta(x)$ with 3 replicates in each case and sample size $n=2500$. Optimal degree $m\sim n^{1/\beta}$ selected according with the smoothness condition. In red dotted line, the corresponding Bernstein polynomial estimator, $B_m^{(1)}(Y_n;x)$ for $\rho_\beta(x)$.}
			\label{fig:grafico3}
		\end{figure}
		
	As a consequence of (\ref{conmbed}), the lengths of the confidence intervals in (\ref{c1bou2}) have the respective orders of magnitude  
	\begin{equation} \sqrt{\rho_\beta(x)} n^{-(\beta-1)/(2\beta)},\quad 1<\beta<3,\qquad \hbox{and} \quad  \sqrt{\rho_\beta(x)} n^{-1/3},\quad 3\leq \beta. \label{lengtb} \end{equation}
		On the other hand, if the probability density $\rho$ is smooth, i.e. $\rho \in C^{2}[0,1]$, it is known (cf. Leblanc \cite{leabia}) that the optimal degree of the Bernstein polynomial with respect to MISE is $m\sim n^{2/5}$ and an asymptotic confidence interval can be constructed. When applied to the densities $\rho_\beta,\ \beta\geq 3$, the length of such a confidence interval has the order of magnitude (see Babu et al. \cite{baappl})
		\begin{equation}
			\sqrt{\frac{\rho_\beta(x)}{r(x)}}n^{-2/5},\quad r(x)=\sqrt{4\pi x(1-x)},\quad \beta\geq 3. \label{lengtba}
		\end{equation}
		This order is better than that in (\ref{lengtb}) for $\beta\geq 3$, whenever $x$ is far away from the boundary of $[0,1]$.  Looking at (\ref{conmbed}), we see that the results in (\ref{lengtb}) and (\ref{lengtba}) complement each other. 
		\section{Conclusions}\label{seconcl}
The main contributions of this paper can be summarized as follows.
\begin{enumerate}[(a)]
	\item We give explicit confidence bands and intervals for a distribution function $F$ concentrated on $[0,1]$ by means of random Bernstein polynomials, and for the derivatives $F^{(k)}, \ k\geq 1$, by using random Bernstein-Kantorovich type operators.  The only smoothness assumption on $F$ or $F^{(k)}, \ k\geq 1$, is continuity.
	
	\item In each case, we choose the simplest estimator, that is, the minimum degreee $m$ of the Bernstein estimator fulfilling the requirements to construct the confidence band or interval, for a given confidence level $1-\alpha$.
	
	\item The lengths of the confidence bands and intervals for $F^{(k)}, \ k\geq 0$, essentially depend on the degree of smoothness of $F^{(k)}$.  In any case, such lengths increase as $k$ increases.  As a counterpart, the degree $m$ of the Bernstein estimator decreases as $k$ increases.
	
	\item As expected, the order of magnitude of the length of the confidence band for $F^{(k)}$ is bigger than that of the confidence interval for $F^{(k)}$, excepting for $k=0$, case in which both orders are $1/\sqrt{n}$.
	
	\item Suppose that $F$ is constant in a certain subinterval $(a,b)$ of $[0,1]$.  To construct the confidence interval for $F(x)$, $x\in (a,b)$, the degree $m$ of the Bernstein estimator can be chosen as $m\sim \log n$.  A similar property holds true for $F^{(1)}$.  In such a case, the length of the corresponding confidence interval has the order of $(\log n/n)^{1/2}$.
	\item In the particular case of the uniform distribution, we consider a second-order Bernstein polynomial.   This estimator is much simpler and outperforms the classical uniform empirical process used in the D-K-W inequality. 
	\end{enumerate}
}
	
\section*{Acknowledgments}
\cs{The authors are very grateful to the referees for their careful reading of the manuscript and their constructive criticism, which greatly improved the final outcome. They also thank the Editorial Team for the additional time given to prepare the revised version of this paper.}

		First and third authors are supported by Gobierno de Aragón Research Project E48\_23R, and second author by the Research Project S41\_23R. The third author is also supported by Spanish Research Project PID2021-123737NB-100 (MINECO/FEDER).

	\appendix
	\section{Appendix.  Proofs of the results.} \label{seappe}
	\begin{proof}[Proof of Theorem \ref{thcoba}]
			Let $k\in \mathbb{N}_0$. From (\ref{berand}) and (\ref{derapp}), we have  
			\begin{equation}
				\left\|B_m^{(k)}(Y_n-F;x)\right\|\leq 2^k(m)_k\left\|\frac{S_n(x)}{n}-x\right\|.  \label{rtermk}
			\end{equation}
			Let $\epsilon=l_{k,\alpha}(m)/\sqrt{n}$.  By (\ref{cowilk}), we see that
			\[	
			\left\|\frac{(m)_k}{m^k} B_{m,k}(F^{(k)};x)-F^{(k)}(x)\right\|\leq \epsilon/2.	\]
			We therefore have from  (\ref{decobd}) and (\ref{rtermk})
			\begin{align}
				&P\left(\left\|B_m^{(k)}(Y_n;x)-F^{(k)}(x)\right\|>\epsilon \right) \leq
				P\left(\left\|B_m^{(k)}(Y_n-F;x)\right\|>\frac{\epsilon}{2} \right)\nonumber\\&\leq P\left(	\left\|\frac{S_n(x)}{n}-x\right\|> \frac{\epsilon}{2 \cdotp 2^k(m)_k}  \right)=P\left(	\left\|\frac{S_n(x)}{n}-x\right\|> \frac{\sqrt{2 \log(2/\alpha)}}{2 \sqrt{n}}  \right),	\end{align}
			where the last equality follows from (\ref{dflkam}). Thus, the result is a consequence of (\ref{expbmg}) and some simple computations.
\end{proof}

\begin{proof} [Proof of Corollary \ref{counba}]  By  (\ref{decobd}) and (\ref{center}), we have
	\begin{equation}
		B_2(Y_n,x)-x=B_2(Y_n-F,x). \label{cente2}
	\end{equation}
	From (\ref{berand}), we see that $B_2(Y_n-F,x)=0$, if $\widetilde{S}_2(x)\in\{0,2\}$.  Thus,
	\begin{equation}
		\|B_2(Y_n-F,x)\|\leq  \left\|\frac{S_n(x)}{n}-x\right\|\sup_{x\in (0,1)}\widetilde{P}\left(\widetilde{S}_2(x)=1\right) =\frac{1}{2}\left\|\frac{S_n(x)}{n}-x\right\|. \label{cente3}
	\end{equation}
	We therefore have from  (\ref{expbmg}),  (\ref{cente2}), and  (\ref{cente3})
	\[P\left(\left\|B_2(Y_n;x)-x\right\|>\delta\right)\leq P\left(\left\|\frac{S_n(x)}{n}-x\right\|>2\delta\right) \leq 2 e
	^{-8n\delta^2},\]
	thus showing (\ref{unidis}).  Again by (\ref{decobd}) and (\ref{center}), we have 
	\[	B_2^{(1)}(Y_n,x)-1=	B_2^{(1)}(Y_n-F,x). \]
	Hence, inequality (\ref{unider}) follows from (\ref{expbmg}) and (\ref{rtermk}). 
\end{proof}
Let $J$ be a subinterval of $[0,1]$.  Let $\mathbb{Z}=(Z(x),\ x\in J)$ be a stochastic process such that $\mathbb{E}Z(x)=0$ and $|Z(x)|\leq d$, for some $d>0$, $x\in J$.  Denote by $\mathbb{Z}_k=(Z_k(x),\ x\in J),\ k \in \mathbb{N}$, a sequence of independent copies of $\mathbb{Z}$ and define
\begin{equation}
	W_n(x)=Z_1(x)+\dots+Z_n(x),\quad  x\in J. \label{dfwm}
\end{equation}	
On the other hand, let $\widetilde{T}$ be a $J$-valued random variable such that
\begin{equation}
	\widetilde{P}\left(\widetilde{T}=x_i\right)=\alpha_i, \quad i=0,1,\dots, r,\quad r\in \mathbb{N}. \label{dfwitp}
\end{equation}	
Finally, we consider the random variable
\begin{equation}
	\widetilde{E}W_n(	\widetilde{T})=\sum_{i=0}^r	W_n(x_i)\alpha_i.\label{dfewit}
\end{equation}
The following concentration inequality is crucial to obtain confidence intervals. Although similar to Lemma 2.1. in Babu et al. \cite{baappl}, it is stated in a form useful to make explicit computations. 
\begin{lem} For any $c\geq d^{-2}\ \mathbb{E}\times \mathbb{\widetilde{E}}\ Z(\widetilde{T})^2$, we have
	\[P\left(|\mathbb{\widetilde{E}}W_n(\widetilde{T})|\geq ncd\epsilon\right)\leq 2 e^{-nc\tau(\epsilon)},\quad \epsilon>0,\] 
	where $\tau(\epsilon)$ is defined in (\ref{dftaue}).\label{lecowt}
\end{lem}
\begin{proof}Let $x\in J$ and $\beta>0$.  We claim that
	\begin{equation}
		\max\left(\mathbb{E}e^{\beta \mathbb{\widetilde{E}}W_n(\widetilde{T}) }, \mathbb{E}e^{-\beta \mathbb{\widetilde{E}}W_n(\widetilde{T}) }\right)\leq exp\left(n \frac{e^{\beta d}-1-\beta d}{d^2}\mathbb{E}\times \mathbb{\widetilde{E}}\ Z(\widetilde{T})^2\right). \label{coewit}
	\end{equation}
	Indeed, we have
	\begin{align*}
		\mathbb{E}e^{\beta Z(x)}&=1+\sum_{j=2}^\infty\frac{\beta^j}{j!} \mathbb{E} Z(x)^j\leq 1+\frac{\mathbb{E}Z(x)^2}{d^2}\sum_{j=2}^\infty\frac{(\beta d)^j}{j!}\\&=1+\frac{e^{\beta d}-1-\beta d}{d^2}\mathbb{E}Z(x)^2\leq exp \left(\frac{e^{\beta d}-1-\beta d}{d^2}\mathbb{E}Z(x)^2\right).
	\end{align*}
	By (\ref{dfwm}), this implies that 
	\[\mathbb{E}e^{\beta W_n(x)}\leq exp \left(n\frac{e^{\beta d}-1-\beta d}{d^2}\mathbb{E}Z(x)^2\right).\]
	Hence, we get from (\ref{dfewit}) and the generalized H\"{o}lder's inequality (cf. Kwon \cite{kwexte})
	\begin{align*}
		&\mathbb{E} e^{\beta \mathbb{\widetilde{E}} W_n(\widetilde{T})}=\mathbb{E}\prod_{i=0}^{r}e^{\beta W_n(x_i)\alpha_i}\leq \prod_{i=0}^{r}\left(\mathbb{E}e^{\beta W_n(x_i)}\right)^{\alpha_i}\\&\leq \prod_{i=0}^{r}exp \left(n\frac{e^{\beta d}-1-\beta d}{d^2}\mathbb{E}Z(x_i)^2  \alpha_i\right)= exp\left(n \frac{e^{\beta d}-1-\beta d}{d^2}\mathbb{E}\times \mathbb{\widetilde{E}}\ Z(\widetilde{T})^2\right).
	\end{align*}
	This shows claim (\ref{coewit}) for the first term on the left-hand side.  The inequality for the second term is shown \cs{by using $-Z(x)$ instead of $Z(x)$}. 
	By Chernoff's inequality, claim (\ref{coewit}), and the inequality $e^{|x|}\leq e^{x}+e^{-x}, \ x \in \mathbb{R}$, we obtain 
	\[P\left(|\mathbb{\widetilde{E}}W_n(\widetilde{T})|\geq ncd\epsilon\right)\leq e^{-\beta ncd \epsilon}\mathbb{E}e^{\beta |\mathbb{\widetilde{E}}W_n(\widetilde{T})| } \leq 2 exp\left(n c\left(e^{\beta d}-1-\beta d(1+\epsilon)\right)\right).\]
	Choosing the value of $\beta$ minimizing the exponent, that is, $e^{\beta d}-1=\epsilon$, we have from (\ref{dftaue})
	\[P\left(|\mathbb{\widetilde{E}}W_n(\widetilde{T})|\geq ncd\epsilon\right)\leq  2 e^{-n c\tau (\epsilon)}.\]
	This completes the proof.
\end{proof}
Inequality (\ref{coewit}) says that the random variable $\mathbb{\widetilde{E}}W_n(\widetilde{T})$ is essentially subgaussian with variance proxy $n\ \mathbb{E}\times \mathbb{\widetilde{E}}\ Z(\widetilde{T})^2$ (cf. Bobkov et al. \cite{bochst}).

\begin{proof}[Proof of Theorem \ref{thcob0}] 
	
	We apply Lemma \ref{lecowt} to the stochastic process $\mathbb{Z}$ defined as
	\begin{equation}
		Z(x)=S_1(x)-x, \quad x\in(0,1). \label{dfst0}
	\end{equation}
	From (\ref{berand}), we see that $B_m(Y_n-F;x)=0$, if $\widetilde{S}_m(x)\in \{0,m\}$.  For this reason, we consider the random variable $\widetilde{T}$ defined as 
	\begin{equation}
		\widetilde{P}\left(\widetilde{T}=F(i/m)\right)=\frac{p_{m,i}(x)}{q_m(x)}, \quad i=1,\dots, m-1,\quad q_m(x)=1-x^m-(1-x)^m. \label{dfwit0}
	\end{equation}
	By (\ref{berand}) and (\ref{dfewit}), this implies that
	\begin{equation}
		B_m(Y_n-F;x)=\frac{q_m(x)}{n} \widetilde{E}W_n(\widetilde{T}). \label{reran1}
	\end{equation}	
	On the other hand, observe that $d=1$. Also, we have from (\ref{dfst0}), (\ref{dfwit0}), and Fubini's theorem
	\begin{align}
		\mathbb{E}\times \mathbb{\widetilde{E}}\ Z(\widetilde{T})^2&=\mathbb{\widetilde{E}}\times \mathbb{E} \ Z(\widetilde{T})^2=\mathbb{\widetilde{E}}\sigma^2 (\widetilde{T})\nonumber\\&=\sum_{i=1}^{m-1}\sigma^2 (F(i/m))\frac{p_{m,i}(x)}{q_m(x)}=\frac{B_m\left(\sigma^2 (F);x\right)}{q_m(x)}. \label{dfvat0}
	\end{align}
		Let $\epsilon>0$. We have from  (\ref{reran1}), (\ref{dfvat0}), and Lemma \ref{lecowt} 
	\begin{align} &P\left(|B_m(Y_n-F;x)|\geq \epsilon B_m\left(\sigma^2 (F);x\right)\right)=P\left(|\widetilde{E}W_n(\widetilde{T})|\geq n \frac{ B_m\left(\sigma^2 (F);x\right)}{ q_m(x)}\epsilon\right)\nonumber\\& \leq 2 exp\left(-n \frac{ B_m\left(\sigma^2 (F);x\right)}{q_m(x)}\tau(\epsilon)\right). \label{cobeF1}\end{align}
	Finally, it follows from (\ref{decobe}) that
	\begin{align*}	&P\left(\left|B_m(Y_n;x)-F(x)\right|\geq \epsilon B_m(\sigma^2(F);x)+ \left|B_m(F;x)-F(x)\right| \right)\\&\leq  P\left(\left|B_m(Y_n-F;x)\right|\geq \epsilon B_m(\sigma^2(F);x)\right).\end{align*}
	This and (\ref{cobeF1}) show the result.\end{proof}

\begin{proof}[Proof of Corollary \ref{counup}]
	In this case, we have $\left|B_m(F;x)-F(x)\right|=0$.  Moreover, a straightforward calculus shows that
	\[B_m(\sigma^2(F);x)=\widetilde{E}\left(\frac{\widetilde{S}_m(x)}{m}\left(1-\frac{\widetilde{S}_m(x)}{m}\right)\right)=\frac{m-1}{m}\sigma^2(x).\]
	From (\ref{dfwit0}), we see that $q_2(x)=2\sigma^2(x)$. Thus, the first inequality in  (\ref{counu1}) follows by choosing $m=2$ and $\epsilon=8\delta$ in Theorem \ref{thcob0}.  The second inequality readily follows from (\ref{dftaue}).
\end{proof}

\begin{proof}[Proof of Corollary \ref{coink0}]   We apply Theorem \ref{thcob0} with
	\begin{equation}
		\epsilon=\frac{l_{\alpha}}{\sigma(F(x))}\frac{1}{\sqrt{n}}\leq 1. \label{ink0ep}
	\end{equation}
By (\ref{ink0c1}), (\ref{ink0c2}) and (\ref{ink0ep}), the length of the confidence interval in (\ref{eqbth2}) is bounded above by 	
	\begin{align}
		&\epsilon B_m(\sigma^2(F);x)+\left|B_m(F;x)-F(x)\right|\nonumber\\
		&\leq \epsilon \sigma^2(F(x))+\epsilon \left|B_m(\sigma^2(F);x)-\sigma^2(F(x))\right|+\left|B_m(F;x)-F(x)\right|\leq \left(2 l_{\alpha}+\frac{1}{l_{\alpha}}\right)\frac{\sigma (F(x))}{\sqrt{n}},\label{ink0bo}
	\end{align}
On the other hand, using (\ref{dftaue}), (\ref{dfqmx}) and (\ref{ink0ep}), the exponent in (\ref{eqbth2}) is bounded by
\begin{align*}
	&	-nB_m\left(\sigma^2 (F);x\right)\tau(\epsilon) \leq -\frac{n\epsilon^2}{3}B_m\left(\sigma^2 (F);x\right)\\ &\leq -\frac{n\epsilon^2}{3}\sigma^2(F(x))+\frac{n\epsilon^2}{3}|B_m(\sigma^2(F);x)-\sigma^2(F(x))|\leq -\frac{l_{\alpha}^2}{3}+\frac{1}{3},
	\end{align*}
	where in the last inequality we have used (\ref{ink0c1}) and (\ref{ink0ep}). Thus, the upper bound in (\ref{eqbth2}) is bounded above by 
	\[2 e^{1/3}e^{-l_{\alpha}^2/3}=\alpha,\]
	as follows from (\ref{dflalph}). This, (\ref{ink0bo}), and Theorem \ref{thcob0} show the result.\end{proof}
	
To prove Theorem \ref{thder}, we start from formulas (\ref{decobd}) and (\ref{derapp}).  In view of (\ref{derapp}), we apply Lemma \ref{lecowt} to the stochastic process $\mathbb{Z}=(Z(x),\ 0\leq x \leq 1-k/m)$ defined by
			\begin{align}
				Z(x)&=\sum_{j=0}^{k-1} {k-1 \choose j}(-1)^{k-1-j}\left(S_1\left(F\left(x+\frac{j+1}{m}\right)\right)-S_1\left(F\left(x+\frac{j}{m}\right)\right)\right.\nonumber\\ &\left.-\left(F\left(x+\frac{j+1}{m}\right)-F\left(x+\frac{j}{m}\right)\right)\right)=G(x)-\Delta_{1/m}^kF(x), \label{dfzder}\end{align} 
			where
			\begin{equation}
				G(x)=\sum_{j=0}^{k-1} {k-1 \choose j}(-1)^{k-1-j}\left(S_1\left(F\left(x+\frac{j+1}{m}\right)\right)-S_1\left(F\left(x+\frac{j}{m}\right)\right)\right). \label{dfgdex}
			\end{equation}
			Obviously, $\mathbb{E}Z(x)=0,\ 0\leq x \leq 1-k/m$.  Other useful properties of $\mathbb{Z}$ are given in the following auxiliary result.
			\begin{lem} Let $0\leq x \leq 1-k/m$.  Then,
				\begin{equation}
					\left|Z(x)\right|\leq d_{k}, 	\label{leprz1b}
				\end{equation}
				where $d_k$ is defined in (\ref{leprz1}).  Moreover,
				\begin{equation}
					\mathbb{E}Z(x)^2\leq \mathbb{E}G(x)^2=\frac{\displaystyle {2(k-1) \choose k-1}}{m}\mathbb{\widetilde{E}}\rho\left(x+\frac{\widetilde{R}_{k-1}+\widetilde{V}}{m}\right),	\label{leprz2}
				\end{equation}
				where the random variables $\widetilde{R}_{k-1}$ and $\widetilde{V}$ are defined in (\ref{dfkan2}) and (\ref{leprz3}).
				\label{leprzx}
			\end{lem}
			\begin{proof}
				Looking at the first equality in (\ref{dfzder}), we see that 
				\begin{equation}	|Z(x)|\leq \max_{0\leq j\leq k-1}{k-1 \choose j}= {k-1 \choose \lfloor k/2\rfloor}. \label{leprz3b}  \end{equation} By the second equality in (\ref{dfzder}), we have
				\begin{equation}
					\mathbb{E}Z(x)^2=	\mathbb{E}G(x)^2-\left(\Delta_{1/m}^kF(x)\right) ^2\leq \mathbb{E}G(x)^2. \label{leprz4}
				\end{equation}
				On the other hand, the probability law of $G(x)$ is given by 
				\[P\left(G(x)={k-1 \choose j}(-1)^{k-1-j}\right)=F\left(x+\frac{j+1}{m}\right)-F\left(x+\frac{j}{m}\right), \ j=0,1,\dots, k-1,\] 
				and 
				\[P\left(G(x)=0\right)=F\left(x\right)+1-F\left(x+\frac{k}{m}\right).\]
				We therefore have \cs{from (\ref{eqfwsm}) }
				\begin{align*}
					\mathbb{E}G(x)^2&=\sum_{j=0}^{k-1} {k-1 \choose j}^2 \left(F\left(x+\frac{j+1}{m}\right)-F\left(x+\frac{j}{m}\right)\right)\\&=\frac{1}{m}
					\sum_{j=0}^{k-1} {k-1 \choose j}^2 \mathbb{\widetilde{E}}\rho\left(x+\frac{j+\widetilde{V}}{m}\right)=\frac{\displaystyle {2(k-1) \choose k-1}}{m}\mathbb{\widetilde{E}}\rho\left(x+\frac{\widetilde{R}_{k-1}+\widetilde{V}}{m}\right),
				\end{align*}
				where the last equality follows from (\ref{leprz3}).  The proof is complete.		\end{proof}
				
				As follows from (\ref{eqfwsm}), inequality (\ref{leprz4}) is sharp, as $m\rightarrow \infty$.
				
					\begin{proof}[Proof of Theorem \ref{thder}]	Let $\mathbb{Z}$ be as in (\ref{dfzder}) and denote
	\begin{equation}
		\widetilde{T}=\frac{\widetilde{S}_{m-k}(x)}{m}.\label{lecoi1}
	\end{equation}
	By (\ref{derapp}), (\ref{dfewit}), and (\ref{lecoi1}), we can write
	\begin{equation}
		B_m^{(k)}(Y_n-F;x)=(m)_k \frac{\widetilde{E}W_n\left(\widetilde{T}\right)}{n}.\label{lecoi2}
	\end{equation}
	We thus have from (\ref{decobd})
		\begin{align} &P\left(\left|B_m^{(k)}(Y_n;x)-F^{(k)}(x)\right|\geq \delta L_{m,k}(\rho;x)+ \left|\frac{(m)_{k}}{m^k}B_{m,k}(F^{(k)};x)-F^{(k)}(x)\right|\right)\nonumber\\&\leq P\left(\left|B_m^{(k)}(Y_n-F;x)\right|\geq \delta L_{m,k}(\rho;x)\right)\leq P\left(\left|\widetilde{E}W_n\left(\widetilde{T}\right)\right|\geq \frac{n}{(m)_{k}}\delta L_{m,k}(\rho;x)\right).\label{thder1}\end{align}
	On the other hand, we have from Fubini's theorem, Lemma \ref{leprzx}, and (\ref{lecoi1})
\begin{align}
	&\frac{1}{d_k^2}	\mathbb{E}\times \mathbb{\widetilde{E}}\left(Z(\widetilde{T})\right)^2=	\frac{1}{d_k^2}	\mathbb{\widetilde{E}}\times \mathbb{E} \left(Z(\widetilde{T})\right)^2\nonumber\\&\leq \frac{\displaystyle{2(k-1) \choose k-1}}{m d_k^2}\widetilde{E}\rho\left(\frac{\widetilde{S}_{m-k}(x)+\widetilde{R}_{k-1}+\widetilde{V}}{m}\right)=\frac{1}{(m)_k d_k r_k(m)}L_{m,k}(\rho;x)=:c,\label{thder2}
\end{align}
	where the last equality follows from (\ref{dfrkm}) and (\ref{dfkan2}).  
	
	To estimate the right-hand side in (\ref{thder1}), we use Lemma \ref{lecowt} with $c$ as given in (\ref{thder2}), $d=d_k$ and $\epsilon=r_k(m)\delta$ to obtain 
	\begin{align*}
		&P\left(\left|\widetilde{\mathbb{E}}W_n(\widetilde{T})\right|\geq \frac{n}{(m)_k }\delta L_{m,k}(\rho;x)\right)=P\left(\left|\widetilde{E}W_n(\widetilde{T})\right|\geq   \frac{n}{(m)_k r_k(m)} L_{m,k}(\rho;x) r_k(m) \delta\right) 	 \\& \leq 2 exp\left(\frac{-n}{(m)_k d_k r_k(m)} L_{m,k}(\rho;x)\tau(r_k(m)\delta)\right).\end{align*}
This, in conjunction with (\ref{thder1}), shows the result.
\end{proof}

	\begin{proof}[Proof of Theorem \ref{thcodf}]
	Looking at the exponent in (\ref{thdere}), we choose
	\begin{equation}
		\delta=l_\alpha\sqrt{\frac{(m)_k d_k}{r_k(m) n \rho(x)}}=\frac{l_\alpha}{\sqrt{n \rho(x)}}s_k(m), \label{thcon21}
	\end{equation}
	where the second equality follows from (\ref{dfrkm}) and (\ref{dfskm}).  By (\ref{dfrkm}) and (\ref{thcon21}), we have
	\begin{equation}
		r_k(m)\delta=l_\alpha\sqrt{\frac{(m)_k r_k(m)d_k}{ n \rho(x)}}=l_\alpha d_k \sqrt{\frac{m}{n \rho(x){2(k-1)\choose k-1}}}\leq 1,\label{thcon22}
	\end{equation}
	as follows from assumption (\ref{thebou1}).  By (\ref{thcon1}), (\ref{thcon2}) and (\ref{thcon21}), the length of the confidence interval in (\ref{thdere}) is bounded above by
	\begin{equation}
		\delta \rho(x)+\delta\left| L_{m,k}(\rho, x)-\rho(x)\right|+\left|\frac{(m)_k}{m^k}B_{m,k}(F^{(k)};x)-F^{(k)}(x)\right|\leq \left(2l_\alpha+\frac{1}{l_\alpha} \right) \sqrt{\frac{\rho(x)}{n}}s_k(m) . \label{thcon23}
	\end{equation}
	From (\ref{dftaue}) and the first equality in (\ref{thcon21}), the exponent in (\ref{thdere}) is bounded above by
	\[-\frac{n r_k(m)}{3(m)_k d_k}\delta^2 L_{m,k}(\rho, x)=-\frac{l_\alpha^2}{3\rho(x)}\left(\rho(x)+\left( L_{m,k}(\rho, x)-\rho(x)\right)\right)\leq -\frac{l_\alpha^2}{3}+\frac{1}{3}, \]
	where the last inequality follows from condition (\ref{thcon1}). Therefore, the right-hand side in (\ref{thdere}) is bounded by
	\[2e^{1/3}e^{-l_\alpha^2/3}=\alpha,\]
	as a consequence of (\ref{dflalph}).  This, together with (\ref{thcon23}) and Theorem \ref{thder}, shows the result.
\end{proof}

\begin{proof}[Proof of Lemma \ref{lebemo}] Fix $0\leq h \leq 1/2$.  Suppose that $0<\beta\leq 1$. The function $f_\beta(x)=F_\beta(x+h)-F_\beta(x),\ 0\leq x\leq 1-h$, is decreasing and therefore attains its maximum at $x=0$.  This shows the equality in (\ref{befimo}). If $\beta>1$, then
	\[\omega(F_\beta;h)\leq \|F_\beta^{(1)}\|h=\beta h.\]
	To show (\ref{besemo}), consider the function
	\[g_\beta(x)=2\left(F_\beta(x)-\frac{1}{2}\left(F_\beta(x+h)+F_\beta(x-h)\right)\right),\quad h\leq x\leq 1-h.\] 
	Assume that $0<\beta\leq 1$.  Since $F_\beta$ is concave, the function $g_\beta$ is nonnegative and satisfies 
	\[g_\beta^{(1)}(x)=2\beta\left(F_{\beta-1}(x)-\frac{1}{2}\left(F_{\beta-1}(x+h)+F_{\beta-1}(x-h)\right)\right)\leq 0,\] 
	as $F_{\beta-1}$ is convex.  Therefore, $g_\beta$ attains its maximum at $x=h$ and 
	\[g_\beta(h)=2h^\beta-(2h)^\beta= \left(2-2^\beta\right)h^\beta.\]
	If $1<\beta\leq 2$, we proceed in a similar way, by considering the nonnegative function $-g_\beta(x)$.  Finally, if $\beta >2$, then 
	\[\omega_2(F_\beta;h)\leq \|F_\beta^{(2)}\|h^2\leq \beta (\beta-1) h^2.\]
	The proof is complete.
	\end{proof}

\end{document}